\documentclass[twoside,11pt]{article}

\usepackage{blindtext}

\usepackage{jmlr2e}

\usepackage{amsmath}
\usepackage{graphicx}
\usepackage{enumerate}
\usepackage{url} 
\usepackage[ruled,linesnumbered]{algorithm2e}
\usepackage{algpseudocode}
\usepackage{hyperref}
\hypersetup{
	colorlinks=true,
	linkcolor=blue,
	citecolor=blue,
}

\usepackage{caption}
\usepackage{subcaption}

\def\E{{\mathbf E}}
\def\P{{\mathbf P}}
\def\Pa{{\mathcal P}}

\def\D{{\mathcal D}}

\def\S{{\cal S}}

\def\I{{\mathbb{I}}}
\def\R{{\mathbb{R}}}

\def\C{{\small\mathcal{C}}}

\newtheorem{assump}{Assumption}

\usepackage{lastpage}

\ShortHeadings{Properties of Generalized Mondrian
Forests}{Zhan et al.}
\firstpageno{1}

\begin{document}

\def\spacingset#1{\renewcommand{\baselinestretch}%
{#1}\small\normalsize} \spacingset{1}



\title{Non-asymptotic Properties of Generalized Mondrian Forests in Statistical Learning}
\author{\name Haoran Zhan \email haoran.zhan@u.nus.edu \\
       \addr Department of Statistics and Data Science, \\National University of Singapore, 117546, Singapore\\
       \AND
       \name Jingli Wang \email jlwang@nankai.edu.cn \\
       \addr School of Statistics and Data Science, KLMDASR, LEBPS, and LPMC, \\Nankai University,
       Tianjin, 300071, China\\
       \AND
       \name Yingcun Xia \email yingcun.xia@nus.edu.sg\\
       \addr Department of Statistics and Data Science, \\National University of Singapore, 117546, Singapore\\
       }

\editor{My editor}

\maketitle

\begin{abstract}

Random Forests have been extensively used in regression and classification, inspiring the development of various forest-based methods. Among these, Mondrian Forests, derived from the Mondrian process, mark a significant advancement. Expanding on Mondrian Forests, this paper presents a general framework for statistical learning, encompassing a range of common learning tasks such as least squares regression, \(\ell_1\) regression, quantile regression, and classification. Under mild assumptions on the loss functions, we provide upper bounds on the regret/risk functions for the estimators and demonstrate their statistical consistency.

\end{abstract}


\section{Introduction}
\label{sec:intro}

Random Forest (RF) \citep{breiman2001random} is a popular ensemble learning technique in machine learning. It operates by constructing multiple decision trees and averaging their predictions to improve accuracy and robustness. Many empirical studies have demonstrated its powerful performance across various data domains (for example, see \cite{liaw2002classification}). The effectiveness of RF is attributed to its data-dependent splitting rule, called CART. Briefly, CART is a greedy algorithm that identifies the best splitting variable and value by maximizing the reduction of training error between two layers. However, this data-dependent splitting scheme complicates the theoretical analysis of RF.

To the best of our knowledge, two papers have made significant contributions to the theoretical analysis of RF. \cite{scornet2015consistency} was the first to establish the statistical consistency of RF when the true regression function follows an additive model. \cite{klusowski2021universal} later demonstrated the consistency of RF under weaker assumptions on the distribution of predictors. However, both studies rely on the technical condition that the conditional mean has an additive structure. 

To gain a deeper understanding of the random forest, people consider modified and stylized versions of RF. One such method is Purely Random Forests (PRF, for example \cite{arlot2014analysis}, \cite{biau2012analysis} and \cite{biau2008consistency}), where individual trees are grown independently of the sample, making them well suited for theoretical analysis.
In this paper, our interest is Mondrian Forest, which is one of PRFs applying Mondrian process in its partitioning. This kind of forest was first introduced by \cite{lakshminarayanan2014mondrian} and has competitive online performance in classification problems compared to other state-of-the-art methods. Inspired by its nice online property, \cite{mourtada2021amf} also studied the online regression theory and classification using Mondrian forest. Besides its impressive online performance, this data-independent partitioning method is also important for offline regression because it achieves a higher consistency rate than other partitioning ways, such as the midpoint-cut strategy; see \cite{mourtada2020minimax}.
In fact, \cite{mourtada2020minimax} showed that the statistical consistency for Mondrian forests is minimax optimal for the class of H\"{o}lder continuous functions. Then, \cite{cattaneo2023inference}, following the approach in  \cite{mourtada2020minimax}, derived the asymptotic normal distribution of the Mondrian forest for the offline regression problem. Recently, \cite{baptista2024trim} proposed a novel dimension reduction method using Mondrian forests.  Therefore, it is evident that Mondrian forests have attracted increasing research attention due to their advantageous theoretical properties compared to other forest-based methods.

In this paper, we argue that Mondrian forests, beyond their application in classical regression and classification, can be extended to address a broader spectrum of statistical and machine learning problems, including generalized regression, density estimation, and quantile regression. Our main contributions are as follows:

\begin{itemize}
    \item First, we propose a general framework based on Mondrian forests that can be applied to various learning problems.

    \item Second, we establish an upper bound for the regret (or risk) function of the proposed forest estimator. The theoretical results derived from this analysis are applicable to numerous learning scenarios, with several examples provided in Section \ref{sec:Examples}.
\end{itemize}


Our training approach adopts a global perspective, contrasting with the local perspective utilized in the generalization of Random Forests by \cite{athey2019generalized}. Specifically, after performing a single optimization on the entire dataset, our method can estimate the objective function \( m(x) , \forall x \in [0,1]^d \), while \cite{athey2019generalized} focus on estimating \( m(x_0) \) at a specific point \( x_0 \). This global approach significantly reduces computational time, especially in high-dimensional settings where \( d \) is large.

Moreover, our globalized framework can be easily applied to statistical problems involving a penalization function \( Pen(m) \), which depends on \( m(x) \) in the entire domain \( [0,1]^d \). In Section \ref{densityestimation}, we illustrate the application of our method to nonparametric density estimation, a problem that incorporates such penalization. In contrast, since \cite{athey2019generalized} relies on pointwise estimation, their method struggles to ensure that the resulting estimator satisfies shape constraints, such as those required for a valid probability density function, thus excluding these cases from their scope. 

\section{Background and Preliminaries}
\subsection{Task in Statistical Learning}
Let $(X,Y)\in [0,1]^d\times \R$ be the random vector, where we have normalized the range of $X$. In statistical learning, the goal is to find a policy $h$ supervised by $Y$, which is defined as a function $h: [0,1]^d\to \R$. Usually, a loss function $\ell (h(x),y): \R\times\R\to [0,\infty)$ is used to measure the difference or loss between the decision $h(x)$ and the goal $y$. Taking expectation w.r.t. $X,Y$, the risk function 
\begin{equation}\label{Risk}
    R(h):=\E(\ell (h(X),Y)
\end{equation}
denotes the averaged loss by using the policy $h$. Naturally, people have reasons to select the best policy $h^*$ by minimizing the averaged loss over some function class $\mathcal{H}_1$, namely,
$$
h^*=\arg\min_{h\in\mathcal{H}_1 }  R(h).
$$
Therefore, the policy $h^*$ has the minimal risk  and is the best in theory. In practice, the distribution of $(X,Y)$ is unknown and \eqref{Risk} is not able to be used for the calculation of $R(h)$. Thus, such a best $h^*$ cannot be obtained in a direct way. Usually, we can use i.i.d. data $\D_n=\{(X_i,Y_i)\}_{i=1}^n$ to approximate $R(h)$ by the law of large numbers. Thus, the empirical risk function can be approximated  by
\begin{equation*}
    \hat{R}(h):=\frac{1}{n}\sum_{i=1}^n \ell (h(X_i),Y_i).
\end{equation*}
Traditionally, people always can find an estimator/policy $\hat{h}_n:[0,1]^d\to \R$ by minimizing $\hat{R}(h)$ over a known function class; see spline regression in \cite{gyorfi2002distribution}, wavelet regression in \cite{gyorfi2002distribution} and regression by deep neural networks in \cite{schmidt2020nonparametric} and \cite{kohler2021rate}. 

Instead of globally minimizing \( \hat{R}(h) \), tree-based greedy algorithms adopt a ``local" minimization approach, which is widely used to construct the empirical estimator \( \hat{h}_n \). These algorithms have demonstrated the advantages of tree-based estimators over many traditional statistical learning methods. Therefore, in this paper, we focus on bounding the regret function:
\[
\varepsilon(\hat{h}_n) := R(\hat{h}_n) - R(h^*),
\]
where \( \hat{h}_n \) represents an ensemble estimator built using Mondrian Forests.

\subsection{Mondrian partitions}
Mondrian partitions are a specific type of random tree partition, where the partitioning of \( [0,1]^d \) is independent of the data points. This scheme is entirely governed by a stochastic process called the Mondrian process, denoted as \( MP([0,1]^d) \). The Mondrian process \( MP([0,1]^d) \), introduced by \cite{roy2008mondrian} and \cite{roy2011computability}, is a probability distribution over infinite tree partitions of \( [0,1]^d \). For a detailed definition, we refer to Definition 3 in \cite{mourtada2020minimax}.

In this paper, we consider the Mondrian partitions with stopping time $\lambda$, denoted by $MP(\lambda, [0,1]^d)$ (see Section 1.3 in \cite{mourtada2020minimax}).  Its construction consists of two steps. First, we construct partitions according to $MP([0,1]^d)$ by iteratively splitting cells at random times, which depends on the linear dimension of each cell. The probability of splitting along each side is proportional to the side length of the cell, and the splitting position is chosen uniformly. Second, we cut the nodes whose birth time is after the tuning parameter $\lambda>0$. In fact, each tree node created in the first step was given a specific birth time. As the tree grows, so does the birth time. Therefore, the second step can be seen as a pruning process, which helps us to choose the best tree model for a learning problem.

To clearly present the Mondrian partition algorithm, some notations are introduced. 
Let $\C=\prod_{j=1}^d{\C^j}\subseteq [0,1]^d$ be a cell with closed intervals $\C^j=[a_j,b_j]$. Denote $|\C^j|=b_j-a_j$ and $|\C|=\sum_{j=1}^{d}{|\C^j|}$. Let $Exp(|\C|)$ be an exponential distribution with expectation $|\C|>0$. 
Algorithm \ref{alg:1} shows how our Mondrian partition operates. This algorithm is a recursive process, where the root node \( [0,1]^d \) and stopping time  $\lambda$ are used as the initial inputs.



\begin{algorithm}
\caption{Mondrian Partition of $[0,1]^d$: Generate a Mondrian partition of $[0,1]^d$, starting from time $0$ and until time $\lambda$.
}\label{alg:1}
\KwIn{Stopping time $\lambda$.}
 Run the iterative function $\text{Mondrian}\text{Partition}([0,1]^d, 0,\lambda)$.

 \# This function is defined in Algorithm \ref{alg:2}.
\end{algorithm}

\begin{algorithm}
\caption{MondrianPartition $(\C,\tau,\lambda)$: Generate a Mondrian partition of cell $\C$, starting from time $\tau$ and until time $\lambda$.
}\label{alg:2}
Sample a random variable $E_\C \sim Exp(|\C|)$

 \eIf{$\tau+E_\C\le\lambda$}
 { Randomly choose a split dimension $J\in\{1, \ldots , d\}$ with $\P(J=j)=(b_j-a_j)/|\C|$\;
  Randomly choose a split threshold $S_J$ in $[a_J,b_J]$\;
   Split $\C$ along the split $(J,S_J)$:  let $\C_0 = \{x\in \C : x_J\le S_J \}$ and $\C_1 = \C /\ 
\C_0$\;
   \textbf{return} $ \text{Mondrian}\text{Partition}(\C_0,\tau+E_\C,\lambda )\cup  \text{Mondrian}\text{Partition}(\C_1,\tau+E_\C,\lambda )$.
 }
 {
      \textbf{return} $\{\C\}$ (i.e., do not split $\C$)
 }
\end{algorithm}

\begin{figure}
    \begin{subfigure}[b]{0.45\textwidth}
    \centering
  \includegraphics[width=0.8\linewidth]{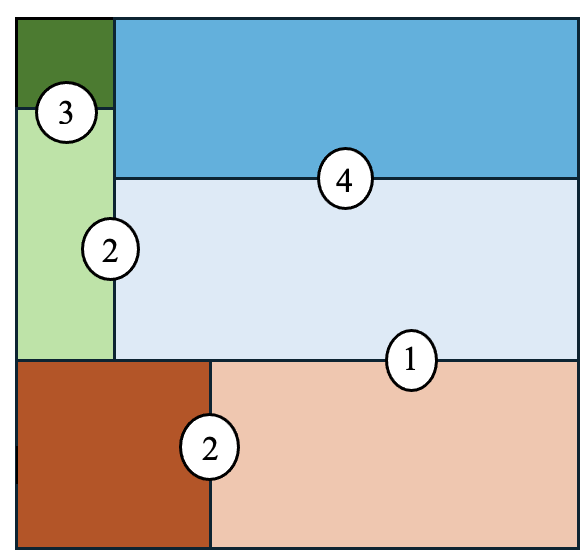}
  \end{subfigure}
       \hfill
   \begin{subfigure}[b]{0.45\textwidth}
    \centering
  \includegraphics[width=0.8\textwidth]{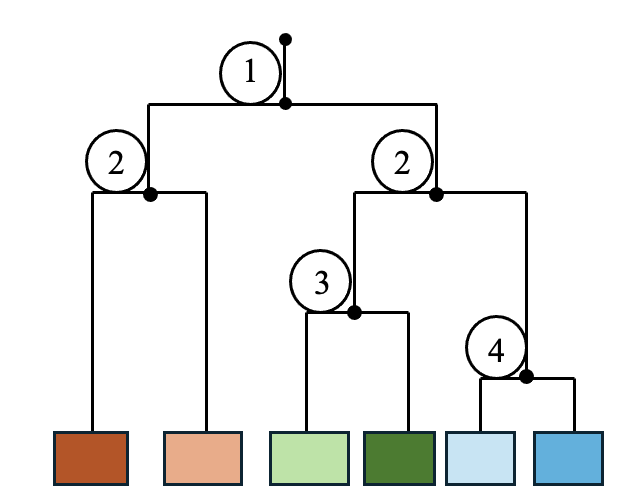}
  \end{subfigure}
  \caption{ An example of a Mondrian partition (left) with the corresponding tree structure (right). This shows how the tree grows over time.
 There are four partitioning times in this demo, $1,2,3,4$, which are marked by bullets ($\bullet$) and the stopping time is $\lambda=4$.}
\label{f1}
\end{figure}

\section{Methodology}
\label{sec:meth}
Based on the Mondrian Partition, we introduce our Mondrian Forests in this section. 
Let $MP_b([0,1]^d),b=1,\ldots,B$, be independent Mondrian processes. When we prune each tree at time $\lambda>0$, independent partitions $MP_b(\lambda, [0,1]^d),b=1,\ldots,B$ are obtained, where all cuts after time $\lambda$ are ignored. In this case, we can write $MP_b(\lambda,[0,1]^d)=\{\C_{b,\lambda,j}\}_{j=1}^{K_b(\lambda)}$ satisfying
$$
[0,1]^d=\bigcup_{j=1}^{K_b(\lambda)} \C_{b,\lambda,j}\ \ \text{and}\ \ \C_{b,\lambda,j_1}\cap \C_{b,\lambda,j_2}= \varnothing,\ \forall j_1\neq j_2,
$$
where $\C_{b,\lambda,j}\subseteq [0,1]^d$ denotes a cell in the partition $MP_b(\lambda,[0,1]^d)$. For each cell $\C_{b,\lambda,j}$, a constant policy $\hat{c}_{b,\lambda,j}\in\R$ is used as the predictor of $h(x)$, where
\begin{equation}\label{dasfASFD}     \hat{c}_{b,\lambda,j}=\arg\min_{z\in [-\beta_n,\beta_n]}{ \sum_{i:X_i\in \C_{b,\lambda,j}}}{ \ell(z,Y_i)}
\end{equation}
and $\beta_n>0$ is a threshold. For any fixed $y\in\R$, $\ell(\cdot,y)$ is usually a continuous function w.r.t. the first variable in machine learning. Therefore, the optimization \eqref{dasfASFD} over $[-\beta_n,\beta_n]$ guarantees the existence of  $ \hat{c}_{b,\lambda,j}$ in general and we allow $\beta_n\to\infty$ as $n\to\infty$ in theory. Then, for each $1\le b\le B$, we can get an estimator of $h(x)$:
$$
\hat{h}_{b,n}(x):=\sum_{j=1}^{K_b(\lambda)}{\hat{c}_{b,\lambda,j}\cdot \I(x\in\C_{b,\lambda,j})},\ x\in [0,1]^d,
$$
where $\I(\cdot)$ denotes the indicator function. By applying the ensemble technique, the final estimator is given by
\begin{equation}\label{final estimator}
    \hat{h}_{n}(x):=\frac{1}{B}\sum_{b=1}^B {\hat{h}_{b,n}(x)}, \ x\in [0,1]^d.
\end{equation}
If the cell $C_{b,\lambda,j}$ does not contain any data point, we just use $0$ as the optimizer in the corresponding region.

Let us clarify the relationship between \eqref{final estimator} and the traditional RF.  Recall that the traditional RF (see \cite{mourtada2020minimax} and \cite{breiman2001random}) is used to estimate the conditional mean $\E(Y|X=x)$ by using the sample mean in each leaf. In this case,  the $\ell_2$ loss function $\ell(v,y)=(v-y)^2$ is  applied.  If $|Y|\le \beta_n$,  it can be checked that 
$$
  \hat{c}_{b,\lambda,j}= \frac{1}{Card(\{i:X_i\in \C_{b,\lambda,j}\})}\sum_{i:X_i\in \C_{b,\lambda,j}} Y_i,
$$
where $Card(\cdot)$  denotes the cardinality of a set.  Therefore,  in this case, our forest estimator  \eqref{final estimator}  coincides with the traditional RF  exactly.  In conclusion,  $\hat{h}_{n}(x)$ is an extension of traditional RF in \cite{breiman2001random} since the loss function $\ell(v,y)$ can be chosen flexibly in different learning problems. 

From the above learning process, we know that there are two tuning parameters in the construction of $\hat{h}_n$, namely $\lambda$ and $B$.  We give some remarks about them. The stopping time, $\lambda$, controls the complexity of the Mondrian forest. Generally speaking, the cardinality of a tree partition increases with the value of $\lambda$. Thus, a large value of $\lambda$ is beneficial in reducing the bias of the forest estimator. In contrast, a small $\lambda$ is instrumental in controlling the generalization error of the final estimator \eqref{final estimator}. So selecting $\lambda$ is crucial in balancing complexity and stability. To ensure the consistency of $\hat{h}_n$, we suppose that $\lambda$ is dependent on the sample size $n$, denoting it as $\lambda_n$ in the following analysis. The second parameter, $B$, denotes the number of Mondrian trees, which can be determined as the selection for RF. There are many studies on its selection for RF; see, for example, \cite{Zhang2009}. In practice, most practitioners take $ B= 100 \text{\ or\ } 500$ in their computations.

\section{Main results}
\label{sec: regret function of Mondrian forests}

In this section, we study the upper bound of the regret function of  \eqref{final estimator}, which is constructed by Mondrian processes.  Denote $\S\subseteq\R$ as the support of $Y$ that satisfies $\P(Y\in\S)=1$. First, we need some mild restrictions on the loss function $\ell(v,y)$.

\begin{assump}\label{assump1}
The risk function $R(h):=\E(\ell(h(X),Y))$ is convex. In other words, for any $\epsilon \in (0,1)$ and functions $h_1,h_2$, we have $R(\epsilon h_1+(1-\epsilon) h_2)\le \epsilon R(h_1)+(1-\epsilon) R(h_2)$.
\end{assump}

Note that Assumption \ref{assump1} is satisfied if $\ell(\cdot,y)$ is convex for any fixed $y\in \S$.

\begin{assump}\label{assump2}
 There is a nonnegative function $M_1(v,y)>0$ with $ v>0, y\in \R$  such that for any $y\in\S$, $\ell(\cdot,y)$ is Lipschitz continuous and for any $v_1,v_2\in [-v,v]$, $y\in\S $, we have
$$
 |\ell(v_1,y)-\ell(v_2,y)|\le M_1(v,y)|v_1-v_2|.
$$
\end{assump}

\begin{assump}\label{assump3}
    There is an envelope function $M_2(v,y)>0$ such that for any $v\in\R^+$ and $y\in\S$,
    $$
    \left|\sup_{v'\in [-v,v]}\ell(v',y)\right| \le M_2(v,y)\ \ \text{and}\ \ \E(M_2^2(v,Y))<\infty.
    $$
\end{assump}

Without loss of generality, we can assume that $M_2(\cdot, y)$ is non-decreasing w.r.t. the first variable for any fixed $y\in\S$. In the next section, we will see that many commonly used loss functions satisfy Assumption \ref{assump1}- \ref{assump3} including $\ell_2$ loss and $\ell_1$ loss. In the theoretical analysis, we make the following Assumption \ref{assump_distribution} on the distribution of $Y$ and assume $X$ takes value in $[0,1]^d$. This paper mainly focuses on the sub-Gaussian case of $Y$; namely $w(t)=t$. But Assumption \ref{assump_distribution} extends the range of sub-Gaussian distributions to a more general class. 

\begin{assump}\label{assump_distribution}
    There is an increasing function $w(t)$ satisfying $\lim_{t\to\infty}w(t)=\infty$ and constant $c>0$, such that
    $$
    \limsup_{t\to\infty}\P(|Y|>t) \exp{(tw(t)/c)}<\infty.
    $$ 
\end{assump}

Our theoretical results relate to the $(p,C)$-smooth class given below 
since it is large enough and dense in the $L^2$ integrable space generated by any probability measure. When $0<p\le 1$, this class is also known as Holder space with index $p$ in the literature; see \cite{adams2003sobolev}. Additionally, the $(p,C)$-smooth class is frequently used in practice, such as the splines, because its smoothness makes the computation convenient. 

\begin{definition}[$(p,C)$-smooth class]\label{def1}
    Let $p=s+\beta> 0$, $\beta\in (0,1]$ and $C>0$. The $(p,C)$-smooth ball with radius $C$, denoted by $\mathcal{H}^{p,\beta}([0,1]^d,C)$, is the set of $s$ times differentiable functions $h: [0,1]^d\to\R$ such that
    $$
     |\nabla^s h(x_1)-\nabla^s h(x_2)|\le C\|x_1-x_2\|_2^\beta, \ \ \forall x_1,x_2\in [0,1]^d.
    $$
    and 
    $$
    \sup_{x\in [0,1]^d}{|h(x)|}\le C,
    $$
    where $\|\cdot\|_2$ denotes the $\ell_2$ norm in $\R^d$ space and $\nabla$ is the gradient operator.
\end{definition}


\begin{theorem}[Regret function bound of Mondrian forests]\label{Th1}
   Suppose that the loss function $\ell(\cdot,\cdot)$ satisfies Assumption \ref{assump1}-\ref{assump3} and that the distribution of $Y$ satisfies Assumption \ref{assump_distribution}. For any $h\in \mathcal{H}^{p,\beta}([0,1]^d,C)$ with $0<p\le 1$, we have
   \begin{align}
       \E R(\hat{h}_n)-R(h)&\le \underbrace{c_1\cdot\frac{\max\{\beta_n,\sqrt{\E(M_2^2(\beta_n, Y))}\}}{\sqrt{n}}(1+\lambda_n)^d}_{generalization \ error}+\underbrace{2d^{\frac{3}{2}p}C \sup_{y\in [-\ln n,\ln n]}{M_1(C,y)}\cdot \frac{1}{\lambda_n^p}}_{approximation \ error}\nonumber\\
       &+ \underbrace{c_2\left(  \sup_{x\in [-\beta_n,\beta_n]}{|\ell(x,\ln n)|}+\sqrt{\E(M_2^2(\beta_n,Y))}+C\sqrt{\E(M_1^2(C,Y))}\right)\cdot \frac{1}{n}}_{residual\  caused \  by\  the \  tail  \ of \  Y},\label{dasjkf}
   \end{align}
   where $c_1,c_2>0$ are some universal constants.
\end{theorem}

   \begin{remark}
       The first term of the RHS of \eqref{dasjkf} relates to the generalization error of the forest, and the second one is the approximation error of Mondrian forest to $h\in \mathcal{H}^{p,\beta}([0,1]^d,C)$.  Finally, the last line is caused by the tail property of $Y$ and will disappear if we assume that $Y$ is bounded. 
   \end{remark}

   \begin{remark}
        We will see that the coefficients above in many applications, such as $\E(M_2^2(\beta_n,Y))$ and $\sup_{y\in [-\ln n,\ln n]}{M_1(C,y)}$, are only of polynomial order of $\ln n$ if $\beta_n\asymp\ln n$. In this case, the last line of \eqref{dasjkf} usually has no influence on the convergence speed of the regret function. Roughly speaking, only the first two terms dominate the convergence rate, namely the generalization and approximation errors. 
   \end{remark}
   
   \begin{remark}\label{remark3}
      If  $\max\{\beta_n,\sqrt{\E(M_2^2(\beta_n, Y))}\}$, $\sup_{y\in [-\ln n,\ln n]}{M_1(C,y)}$, $ \sup_{x\in [-\beta_n,\beta_n]}{|\ell(x,\ln n)|}$ diverge no faster than $O((\ln n)^\gamma)$ for some $\gamma>0$,  we know from \eqref{dasjkf} that
      $$
       \varlimsup_{n\to\infty} \E(R(\hat{h}_n))\le \inf_{h\in \mathcal{H}^{p,\beta}([0,1]^d,C)} R(h)
      $$
      when $\lambda_n\to\infty$ and $\lambda_n=o(n^{\frac{1}{2d}})$. Therefore, Mondrian forests perform not worse than $(p,C)$-smooth class in the general setting. 
   \end{remark}

In statistical learning, the consistency of an estimator is a crucial property that ensures that the estimator converges to the true value of the parameter/function being estimated as the sample size increases. Theorem \ref{Th1} can also be used to analyze the statistical consistency of $\hat{h}_n$. For this purpose,  denote the true function $m$ by 
\begin{equation}\label{true}
  m:=\arg min_{\forall g} {R(g)},
\end{equation}
 and make the following assumption.

\begin{assump}\label{assump5}
For any $h:[0,1]^d\to [-\beta_n,\beta_n]$, there are $c>0$ and $\kappa\ge 1$ such that
 $$
   c^{-1}\cdot \E|h(X)-m(X)|^\kappa \le   R(h)-R(m)\le  c\cdot \E|h(X)-m(X)|^\kappa.
$$
\end{assump}

Usually, $\kappa=2$ holds in many specific learning problems. Before presenting the consistency results, we denote the last line of \eqref{dasjkf} by $Res(n)$, namely,
\begin{equation}\label{res}
   Res(n):= c_2\left(  \sup_{x\in [-\beta_n,\beta_n]}{|\ell(x,\ln n)|}+\sqrt{\E(M_2^2(\beta_n,Y))}+C\sqrt{\E(M_1^2(C,Y))}\right)\cdot \frac{1}{n}.
\end{equation}

Then, the statistical consistency of Mondrian forests can be guaranteed by the following two corollaries. 

\begin{corollary}[Consistency rate of  Mondrian forests]\label{Th2}
    Suppose that the loss function $\ell(\cdot,\cdot)$ satisfies Assumption \ref{assump1}-\ref{assump3} and that the distribution of $Y$ satisfies Assumption \ref{assump_distribution}. Suppose that the true function  $m \in \mathcal{H}^{p,\beta}([0,1]^d,C)$  with $0<p\le 1$ and Assumption \ref{assump5} is satisfied. Then,
    \begin{align}
       \E\left| \hat{h}_n(X) - m(X)\right|^\kappa &\le c_1\cdot\frac{\max\{\beta_n,\sqrt{\E(M_2^2(\beta_n, Y))}\}}{\sqrt{n}}(1+\lambda_n)^d \nonumber\\
       &+ 2d^{\frac{3}{2}p}C \sup_{y\in [-\ln n,\ln n]}{M_1(C,y)}\cdot \frac{1}{\lambda_n^p}
       + Res(n),\label{dsadfsad}
   \end{align}
   where $c_1,c_2>0$ are some universal constants.
\end{corollary}

\begin{corollary}[Consistency of  Mondrian forests]\label{Corro_consistency}
     Suppose the loss function $\ell(\cdot,\cdot)$ satisfies Assumption \ref{assump1}-\ref{assump3} and the distribution of $Y$ satisfies Assumption \ref{assump_distribution}. Suppose  $m(X)$ is  $L^\kappa$ integrable on $[0,1]^d$ and $\E(\ell^2(m(X),Y))<\infty$. Furthermore, Assumption \ref{assump5} is satisfied. If  $\lambda_n=o\left(\left(\frac{\sqrt{n}}{\max\{\beta_n,\sqrt{\E(M_2^2(\beta_n, Y))}\}}\right)^{\frac{1}{d}}\right)$, ${\lambda_n^{-1}}\cdot\sup_{y\in [-\ln n,\ln n]}{M_1(C,y)} \to 0$ and $Res(n)\to 0$, we have
     $$
\lim_{n\to\infty}\E\left| \hat{h}_n(X)-m(X)\right|^\kappa=0.
     $$
\end{corollary}

\section{Model selection: the choice of $\lambda_n$}
\label{sec: choice of lamda}
In practice, the best $\lambda_n$ is always unknown, thus a criterion is necessary in order to stop the growth of Mondrian forests. Otherwise, the learning process will be overfitted. Here, we adopt a penalty methodology as follows. For each $1\le b\le B$, define 
\begin{equation}\label{asdasaad}
     Pen(\lambda_{n,b}):= \frac{1}{n}\sum_{i=1}^n{\ell(\hat{h}_{b,n}(X_i), Y_i)}+\alpha_n\cdot\lambda_{n,b},
\end{equation}
where the parameter $\alpha_{n}>0$ controls the power of penalty and $\hat{h}_{b,n}$ is constructed already in Section \ref{sec:meth} and using   process $MP_b(\lambda_n, [0,1]^d)$. Then, the best $\lambda_{n,b}^*$ is chosen by 
$$
\lambda_{n,b}^*:= \arg min_{\lambda\ge 0}   Pen(\lambda).
$$
Denote $\hat{h}_{b,n}^*$ as the tree estimator that is constructed by the Mondrian process $MP_b(\lambda_{n,b}^*, [0,1]^d)$. Then, our forest estimator based on model selection is given by
\begin{equation}\label{final estimator2}
    \hat{h}_{n}^*(x):=\frac{1}{B}\sum_{b=1}^B {\hat{h}_{b,n}^*(x)}, \ x\in [0,1]^d.
\end{equation}

\begin{theorem}\label{thm3}
     Suppose the loss function $\ell(\cdot,\cdot)$ satisfies Assumption \ref{assump1}-\ref{assump3}. Meanwhile, suppose the distribution of $Y$ satisfies Assumption \ref{assump_distribution}. For any $h\in \mathcal{H}^{p,\beta}([0,1]^d,C)$ with $0<p\le 1$ and $0<\alpha_n\le 1$, we have
   \begin{align}\label{dasjkdadf}
       \E R(\hat{h}_n^*)-R(h)&\le \underbrace{c_1\cdot \frac{\max\{\beta_n,\sqrt{\E(M_2^2(\beta_n, Y))}\}}{\sqrt{n}} \left(1+\frac{ \sup_{y\in [-\ln n,\ln n]}M_2(\beta_n,y)}{\alpha_n}\right)^d}_{generalization \ error} \nonumber\\
        &+\underbrace{(2 d^{\frac{3}{2}p}C\cdot\sup_{y\in [-\ln n,\ln n]}{M_1(C,y)})\cdot \left(\alpha_n\right)^{\frac{p}{2}}}_{approximation\ error}+Res(n),
  \end{align}
 where $c_1,c_2>0$ are some universal constants and $Res(n)$ is defined in \eqref{res}.
\end{theorem}

 By properly choosing the penalty strength $\alpha_n$, we can obtain a convergence rate of the regret function of Mondrian forests according to \eqref{dasjkdadf}. Theorem \ref{thm3} also implies that the estimator \eqref{final estimator2} is adaptive to the smooth degree of the true function $m$. If $p$ is large, this rate will be fast; otherwise, we will have a slower convergence rate. This coincides with the basic knowledge of function approximation. The applications of Theorem \ref{thm3} are given in the next section, where some examples are discussed in detail. In those cases, we will see that coefficients in \eqref{dasjkdadf}, such as $M_1(\beta_n,\ln n)$, can be upper bounded by a polynomial of $\ln n$  indeed. 
 
 \section{Examples}
 \label{sec:Examples}

In this section, we show how to use Mondrian forests in different statistical learning problems. Meanwhile, theoretical properties of these forest estimators, namely $\hat{h}_n(x)$ in \eqref{final estimator} and $\hat{h}_n^*(x)$ in \eqref{final estimator2}, are given based on Theorem \ref{Th1} \& \ref{thm3} for each learning problem. Sometimes, the Lemma \ref{Lemma_auxi} below is useful for verifying the Assumption \ref{assump5}. The proof of this result can be directly completed by considering the Taylor expansion of the real function $R(h^*+\alpha h),\alpha\in [0,1]$ around the point $\alpha=0$.

\begin{lemma}\label{Lemma_auxi}
    For any $h:[0,1]^p\to\R$ and $\alpha\in [0,1]$, we have
    $$
       C_1\E(h(X)^2)\le \frac{d^2}{d\alpha^2} R(h^*+\alpha h) \le  C_2\E(h(X)^2),
    $$
    where constants $C_1>0,C_2>0$ are universal. Then, Assumption \ref{assump5} holds with $\kappa=2$. 
\end{lemma}

\subsection{Least squares regression}\label{subsec:ols}
As shown in \cite{mourtada2020minimax}, Mondrian forests are statistically consistent if the $\ell_2$ loss function  is used. In our first example, we revisit this case by using the general results established in Section \ref{sec: regret function of Mondrian forests}
\& \ref{sec: choice of lamda}.  Usually, nonparametric least squares regression refers to methods that do not assume a specific parametric form of the conditional expectation $\E(Y|X)$. Instead, these methods are flexible and can adapt to the underlying structure of the data. The loss function of the least squares regression is given by $\ell(v,y)=(v-y)^2$. First, we define the event $A_n:=\{ \max_{1\le i\le n}{|Y_i|}\le \ln n\}$. Under  Assumption \ref{assump_distribution}, by \eqref{subgaussguanjian} we can find constants $c,c'>0$ such that $\P(A_n)\ge 1-c'\cdot n e^{-c\ln n\cdot w(\ln n)}$. This means $\P(A_n)$ is very close to $1$ as $n\to\infty$. When the event $A_n$ occurs, from \eqref{dasfASFD} we further know
\begin{align*}
    \hat{c}_{b,\lambda,j}&=\arg\min_{z\in [-\beta_n,\beta_n]}{ \sum_{i:X_i\in \C_{b,\lambda,j}}}{ \ell(z,Y_i)}\\
    &= \frac{1}{Card(\{i:X_i\in \C_{b,\lambda,j}\})}\sum_{i:X_i\in \C_{b,\lambda,j}} Y_i,
\end{align*}
where $Card(\cdot)$  denotes the cardinality of any set. Therefore, $  \hat{c}_{b,\lambda,j}$ is just the average of $Y_i$s that are in the leaf $\C_{b,\lambda,j}$. 

Let us discuss the property of $\ell(v,y)=(v-y)^2$. First, it is obvious that Assumption \ref{assump1} holds for this $\ell_2$ loss. By some simple calculations, we also know Assumption \ref{assump1} is satisfied with $M_1(v,y)=2(|v|+|y|)$ and Assumption \ref{assump2} is satisfied with $M_2(v,y)=2(v^2+y^2)$. Choosing $\lambda_n= n^{\frac{1}{2(p+d)}}$ and $\beta_n\asymp \ln n$, Theorem \ref{Th1} implies the following property of $\hat{h}_n$.
\begin{proposition}
For any $h\in \mathcal{H}^{p,\beta}([0,1]^d,C)$, there is an integer $n_1(C)\ge 1$ such that for any $n>n_1(C)$,
     \begin{equation*}
       \E R(\hat{h}_n)-R(h)\le \left(2\sqrt{2}\ln^2 n+4d^{\frac{3}{2}p}C\cdot (C+\ln n)+1\right)\cdot \left(\frac{1}{n}\right)^{\frac{1}{2}\cdot\frac{p}{p+d}}.
   \end{equation*}
\end{proposition}
Then, we check the consistency as in Corollary \ref{Th2}. By some calculations, we have $m(x)=\E(Y|X=x)$  and 
\begin{align*}
   R(h)-R(m)&= \E(Y-h(X))^2- \E(Y-m(X))^2\\
   &= \E(h(X)-m(X))^2. 
\end{align*}
The above inequality shows that Assumption \ref{assump5} holds with $c=1$ and $\kappa=2$. When $\lambda_n= n^{\frac{1}{2(p+d)}}$, Corollary \ref{Th2} implies Proposition \ref{pro2}.
\begin{proposition}\label{pro2}
For any $h\in \mathcal{H}^{p,\beta}([0,1]^d,C)$, there is an integer $n_2(C)\ge 1$ such that for any $n>n_2(C)$,
    $$
\E\left( \hat{h}_n(X)-m(X)\right)^2\le \left(2\sqrt{2}\ln^2 n+4d^{\frac{3}{2}p}C\cdot (C+\ln n)+1\right)\cdot \left(\frac{1}{n}\right)^{\frac{1}{2}\cdot\frac{p}{p+d}}.
    $$
\end{proposition}
We can also show that $\hat{h}_n$ is statistically consistent for any general function $m$ defined in \eqref{true} when $\lambda_n$ is chosen properly as stated in Corollary \ref{Corro_consistency}. Finally, by choosing $\alpha_n=n^{-\frac{p}{2p+4d}}$ and $\beta_n\asymp \ln n$ in Theorem \ref{thm3},  the regret function of the estimator $\hat{h}^*_n$, which is based on the model selection in \eqref{asdasaad}, has an upper bound as shown in the following proposition.
\begin{proposition}
    For any $h\in \mathcal{H}^{p,\beta}([0,1]^d,C)$, there is an integer $n_3(C)\ge 1$ such that for any $n>n_3(C)$,
   \begin{equation*}
       \E R(\hat{h}_n^*)-R(h)\le c(d,p,C)\cdot\left(\frac{1}{n}\right)^{\frac{1}{2}\cdot\frac{p}{p+2d}}\ln^{2d+1} n,
  \end{equation*}
  where $c(d,p,C)>0$   depends on $d,p,C$ only.
\end{proposition}

\subsection{Generalized regression}
Generalized regression refers to a broad class of regression models that extend beyond the traditional ordinary least squares (OLS) regression, accommodating various types of response variables and relationships between predictors and responses. Usually, in this model, the conditional distribution of $Y$ given $X$ follows an exponential family of distribution
\begin{equation}\label{estimate the unknown functio}
   \P(Y\le y|X=x)=\int_{-\infty}^y\exp\{ B(m(x))v-D(m(x)))\}d\Psi(v),
\end{equation}
where $\Psi(\cdot)$ is a positive measure defined on $\R$,  $\Psi(\R)>\Psi(\{y\})$ for any $y\in\R$, and  function $ D(\cdot)=\ln \int_\R {\exp\{B(\cdot)y\}}\Psi(dy)$ is defined on an open interval $\mathcal{I}$ of $\R$, which is used for the aim of normalization. Now, we suppose the function $A(\cdot):=D'(\cdot)/B'(\cdot)$ exists and  we have $\E(Y|X=x)=A(m(x))$ by some calculations. Thus, the conditional expectation $\E(Y|X=x)$ will be known if we can estimate the unknown function $m(x)$. More information about model \eqref{estimate the unknown functio} can be found in   \cite{stone1986dimensionality}, \cite{stone1994use}   and \cite{huang1998functional}.

The problem of generalized regression is to estimate the unknown function $m(x)$ by using the i.i.d. data $\D_n=\{(X_i,Y_i)\}_{i=1}^n$. Note that both $B(\cdot)$ and $D(\cdot)$ are known in \eqref{estimate the unknown functio}. In this case, we use the maximal likelihood method for the estimation, and the corresponding loss function is given by
$$
 \ell(v,y)=-B(v)y+D(x)
$$
and by some calculations we know the true function $m$ satisfies Definition \ref{true}, namely
$$
m\in \arg min_{h}{\E(-B(h(X))Y+D(h(X)))}.
$$
Therefore, we have reasons to believe that Mondrian forests are statistically consistent in this problem, which is stated in Corollary \ref{Corro_consistency}. Now, we give some mild restrictions on $B(\cdot)$ and $D(\cdot)$ to ensure our general results in Section \ref{sec: regret function of Mondrian forests} \& \ref{sec: choice of lamda} can be applied in this generalized regression.
\begin{enumerate}
  \item [(i)] $B(\cdot)$ has the second continuous derivative, and its first derivative is strictly positive on $\mathcal{I}$.
  \item [(ii)] We can find an interval $S$ of $\R$ satisfying the measure $\Psi$ is concentrated on $S$ and
         \begin{equation}\label{33}
          -B''(\xi)y+D''(\xi)>0, \quad\quad y\in  \breve{S},\xi\in\mathcal{I} 
         \end{equation}
         where $\breve{S}$ denotes the interior of $S$. If $S$ is bounded, \eqref{33} holds for at least one of the endpoints. 
  \item  [(iii)] $\P(Y\in S)=1$ and $\E(Y|X=x)=A(m(x))$ for each $x\in [0,1]^d$.
  \item  [(iv)] There is a compact interval $\mathcal{K}_0$ of $\mathcal{I}$ such that the range of $m$ is contained in  $\mathcal{K}_0$. 
\end{enumerate}

The above restrictions on $B(\cdot)$ and $D(\cdot)$ were used in \cite{huang1998functional}. In fact, we know from \cite{huang1998functional} that many commonly used distributions satisfy these conditions, including the normal distribution, the Poisson distribution, and the Bernoulli distribution. Now, let us verify our Assumption \ref{assump1}-\ref{assump5} under this setting. 

In particular, Assumption \ref{assump1} is verified using restrictions (i)-(iii). On the other hand, we choose the Lipchitz constant in Assumption \ref{assump2} by 
$$
 M_1(v,y):=\left| \sup_{\tilde{v}\in [-v,v]}{B'(\tilde{v})}\right|\cdot |y|+ \left| \sup_{\tilde{v}\in [-v,v]}{D'(\tilde{v})}\right|.
$$
Thirdly, the envelope function of $\ell_2(v,y)$ can be set by
$$
 M_2(v,y):= \left| \sup_{\tilde{v}\in [-v,v]}{B(\tilde{v})}\right|\cdot |y|+ \left| \sup_{\tilde{v}\in [-v,v]}{D(\tilde{v})}\right|\cdot |y|.
$$ 
Since  $Y$ is a sub-Gaussian random variable in Assumption \ref{assump_distribution}, thus $\E(M_2^2(X,Y))<\infty$ in this case,  indicating that Assumption \ref{assump3} is satisfied. Finally, under restrictions (i)-(iv), Lemma 4.1 in \cite{huang1998functional} shows that our Assumption \ref{assump5} holds with $\kappa=2$ and
\begin{equation}
c=\max\left\{ \sup_{\substack{\xi\in [-\beta_n,\beta_n]\cap \mathcal{I}\\ m\in \mathcal{K}_0}}{(-B''(\xi)A(m)+D''(\xi))}], \left[ \inf_{\substack{\xi\in [-\beta_n,\beta_n]\cap \mathcal{I}\\ m\in \mathcal{K}_0}}{(-B''(\xi)A(m)+D''(\xi))}]\right]^{-1} \right\}. \label{ln n}
\end{equation}
From \eqref{33}, the constant $c$ in \eqref{ln n} must be larger than zero. On the other hand, as we will see later, the above $c$ does not equal infinity in many cases.

Therefore, those general theoretical results in Section \ref{sec: regret function of Mondrian forests} \& \ref{sec: choice of lamda} can be applied in generalized regression. Meanwhile, we need to stress that the coefficients in the general results, such as $M_1(\beta_n,\ln n)$ in Theorem \ref{Th1} and $c$ in \eqref{ln n}, are always of polynomial order of $\ln n$ if this $\beta_n$ is selected properly.  Let us give some specific examples.
\begin{enumerate}[(1)]
           \item The first example is Gaussian regression, where the conditional distribution $Y|X=x$ follows $N(m(x),\sigma^2)$ and $\sigma^2$ is known.  Therefore, $B(x)=x$, $D(x)=\frac{1}{2}x^2$, $\mathcal{I}=\R$ and $S=\R$. Our goal is to estimate the unknown conditional mean $m(x)$. Now, restrictions (i)-(iii) are satisfied. To satisfy the fourth one, we assume the range of $m$ is contained in a compact set of $\R$, denoted by $\mathcal{K}_0$. Choose $\beta_n\asymp \ln n$. Meanwhile, we can ensure $Y$ is a sub-Gaussian random variable. The constant $c$ in \eqref{ln n} equals $1$, and those coefficients in general theoretical results are all of polynomial order of $\ln n$, such as $M_1(\beta_n,\ln n)\asymp 2\ln n$.
               
           \item The second example is Poisson regression, where the conditional distribution $Y|X=x$ follows $Poisson(\lambda(x))$ with $\lambda(x)>0$. Therefore, $B(x)=x$, $D(x)=-\exp( x)$, $\mathcal{I}=\R$ and $S=[0,\infty)$. Our goal is to estimate the $\ln n$ transformation of conditional mean, namely $m(x)=\ln \lambda(x)$ in \eqref{estimate the unknown functio}, by using Mondrian forest \eqref{final estimator}. It is not difficult to show restrictions (i)-(iii) are already satisfied. To satisfy the fourth one, we assume the range of $\lambda(x)$ is contained in a compact set of $(0,\infty)$. Thus, $m(x)$ satisfies restriction (iv). In this case, we choose $\beta_n\asymp \ln\ln n$. Note that $W$ satisfies Assumption \ref{assump_distribution} with $w(t)=\ln(1+t/\lambda)$ if $W$ follows Poisson distribution with mean $\lambda>0$. Thus, we can ensure $Y$ satisfies Assumption \ref{assump_distribution}. The constant $c$ in \eqref{ln n} equals to $\ln n$, and those coefficients in general theoretical results are also all of the polynomial order of $\ln n$, such as $M_1(\beta_n,\ln n)\asymp 2\ln n$.

           \item The third example is related to $0-1$ classification, where the conditional distribution $Y|X=x$ follows Bernoulli distribution (taking values in $\{0,1\}$) with $\P(Y=1|X=x)=p(x)\in (0,1)$. It is well known that the best classifier is called the Bayes rule,
               $$
                 C^{Bayes}(x)= 
                \left\{ 
              \begin{array}{lc}
                 1, & p(x)-\frac{1}{2}\ge 0 \\
                    0, & p(x)-\frac{1}{2}<0.\\
                 \end{array}
                \right.
               $$
         What we are interested is to estimate the conditional probability $p(x)$ above. Here, we use Mondrian forest in the estimation. First, we make a shift of $p(x)$, which means $m(x):=p(x)-\frac{1}{2}\in (-\frac{1}{2},\frac{1}{2})$ is used in \eqref{estimate the unknown functio} instead. By some calculations, $B(x)= \ln(0.5+x)-\ln(0.5-x)$, $D(x)= -\ln (0.5-x)$, $\mathcal{I}= (-0.5,0.5)$ and $S=[0,1]$ in this case. Now, the final goal is to estimate $m(x)$ by using the forest estimator \eqref{final estimator}. It is not difficult to show restrictions (i)-(iii) are already satisfied. To satisfy the fourth one, we assume the range of $m(x)$ is contained in a compact set of $(-\frac{1}{2},\frac{1}{2})$. Now, choose $\beta_n\asymp \frac{1}{2}-(\frac{1}{\ln n})^\gamma$ for some $\gamma>0$. Meanwhile, Assumption \ref{assump_distribution} is satisfied since $Y$ is bounded. The constant $c$ in \eqref{ln n} equals to $(\ln n)^{2\gamma}$ and those coefficients in general theoretical results are also all of the polynomial order of $\ln n$, such as $M_1(\beta_n,\ln n)\asymp 2(\ln n)^{\gamma+1}$.

         \item   The fourth example is geometry regression, where the model is $\P(Y=k|X=x)=p(x)(1-p(x))^{k-1}, k\in \mathbb{Z}^+,x\in [0,1]^d$. Here, $p(x)$ denotes the successful probability, and we suppose it is bounded from up and below, namely $p(x)\in [c_1,c_2]\subseteq (0,1)$. Thus, we have a positive probability of obtaining success or failure for any $x$ and $k\in\mathbb{Z}^+$. In this case, $B(x)=x$, $D(x)=-\ln(e^{-x}-1)$ and $m(x)=\ln(1-p(x))$ is the unknown function we need to estimate. Since $m(x)<0$, in this example we optimize  \eqref{dasfASFD} over $[-\beta_n,-\beta_n^{-1}]$ only  with $\beta_n\to\infty$.  In Section \ref{sec: regret function of Mondrian forests}, we only need to replace $[-v,v]$, $\mathcal{H}^{p,\beta} ([0,1]^d,C)$ by $[-v,v]$ and $\mathcal{H}^{p,\beta}([0,1]^d,C)\cap\{h(x):h(x)<0\}$ respectively. Then, it is not difficult to check all results in Section \ref{sec: regret function of Mondrian forests} still hold. Furthermore, restrictions (i)-(iii) are satisfied after some calculation. Finally, we check Assumption \ref{assump5}. In this circumstance, we can still use Lemma 4.1 in \cite{huang1998functional} after replacing $[-\beta_n,\beta_n]$ in \eqref{ln n} with $[-\beta_n,-\beta_n^{-1}]$. If  $\beta_n\asymp \ln\ln n$ is chosen,  it is known $c\asymp (\ln\ln n)^2$ in \eqref{ln n} after some calculation. Therefore, Assumption \ref{assump5} also holds with $c\asymp (\ln\ln n)^2$ and $\kappa=2$.
   
     \end{enumerate}
     
     In each of the four examples above, the convergence rate of $\varepsilon(\hat{h}_n)$ is $O_p(n^{-\frac{1}{2}\cdot\frac{p}{p+d}})$ up to a polynomial of $\ln n$.

\subsection{Huber's loss}
Huber loss, also known as smooth $L^1$ loss and proposed in \cite{huber1992robust}, is a loss function frequently used in regression tasks, especially in machine learning applications. It combines the strengths of both Mean Absolute Error (MAE) and Mean Squared Error (MSE), making it more robust to outliers than MSE while maintaining smoothness and differentiability like MSE. Huber loss applies a quadratic penalty for small errors and a linear penalty for large ones, allowing it to balance sensitivity to small deviations and resistance to large, anomalous deviations. This makes it particularly effective when dealing with noisy data or outliers. In detail, this loss function takes the form
\[
\ell(v,y)=\begin{cases}

\frac{1}{2}(v-y)^2 &|v-y|\le \delta_n,\\

\delta_n(|v-y|-\frac{1}{2}\delta_n) & |v-y|\ge \delta_n.
\end{cases}
\]
This case is interesting and different from commonly used loss functions because such $\ell$ depends on $\delta_n$ and can vary according to the sample size $n$. Although this loss function is a piecewise function which is different with classical loss function, the non-asymptotic result in Theorem \ref{Th1} can still be applied here. 

Let us verify Assumption \ref{assump1}-\ref{assump3} for this loss. Firstly, this loss function is convex and Assumption \ref{assump1} is satisfied. Secondly, for any $y\in\R$ we know that $\ell(\cdot,y)$ is a Lipschitz function with Lipschitz constant $M_1(v,y)= \delta_n$.  Thirdly, we can define 
\[
M_2(v,y)=\begin{cases}

\frac{1}{2}(|v|+|y|)^2 &|v|+|y|\le \delta_n,\\

\delta_n(|v|+|y|-\frac{1}{2}\delta_n) & |v|+|y|\ge \delta_n.
\end{cases}
\]
Since $Y$ is sub-Gaussian by Assumption \ref{assump5}, thus $\E(M_2^2(v,Y))<\infty$ and Assumption \ref{assump3} is satisfied. Finally, coefficients in Theorem \ref{Th1}:
\begin{equation*}
  \max\{\beta_n,\sqrt{\E(M_2^2(\beta_n, Y))}\}, \sup_{y\in [-\ln n,\ln n]}{M_1(C,y)}, \sup_{x\in [-\beta_n,\beta_n]}{|\ell(x,\ln n)|}
\end{equation*}
diverge no faster than a polynomial of $\ln n$ if we take $\beta_n=O(\delta_n)$ and the threshold satisfies $\delta_n=o(\ln^{1+\eta} n)$ for any $\eta>0$. Under the above settings, we can obtain a fast convergence rate of the forest estimator by Theorem \ref{Th1}. 

On the other hand, we aim to show our forest estimator performs better than any $(p,C)-$smooth function even if $\delta_n=o(n^{\frac{1}{2}-\nu})$ for any small $\nu>0$. The reason is given below. By calculation, we have 
$$
\max\{\beta_n,\sqrt{\E(M_2^2(\beta_n, Y))}\}\le \beta_n+ \E(\beta_n+|Y|)^2+ \delta_n \E(\beta_n+|Y|)=O(\beta_n^2+\delta_n \beta_n).
$$
Since $\beta_n\asymp\ln n$, if $\delta_n=o(n^{\frac{1}{2}-\nu})$ and $\lambda_n$ is properly selected by \eqref{dasjkf} we know
 $$
       \varlimsup_{n\to\infty} \E(R(\hat{h}_n))\le \inf_{h\in \mathcal{H}^{p,\beta}([0,1]^d,C)} R(h).
$$

Finally, we show that Huber loss can be applied to estimate the conditional expectation, and the corresponding statistical consistency result is given below.

\begin{proposition}\label{secpro:huber}
  Recall the conditional expectation $m(x):=\E(Y|X=x),x\in [0,1]^d$. If $\E(m^2(X))<\infty$, $\beta_n\asymp\ln n$ and $\delta_n=C\ln n$ for a large $C>0$, we can find a series of $\lambda_n\to\infty$ such that
  $$
  \lim_{n\to\infty}\E(\hat{h}_n(X)-m(X))^2=0.
  $$
\end{proposition}

\subsection{Quantile regression}
Quantile regression is a type of regression analysis used in statistics and econometrics that focuses on estimating conditional quantiles (such as the median or other percentiles) of the distribution of the response variable given a set of predictor variables. Unlike ordinary least squares (OLS) regression, which estimates the mean of the response variable conditional on the predictor variables, quantile regression provides a more comprehensive analysis by estimating the conditional median or other quantiles.

Specifically, suppose that $m(x)$ is the $\tau$-th quantile ($0<\tau<1$) of the conditional distribution of $Y|X=x$. Our interest is to estimate $m(x)$ by using i.i.d. data $\D_n=\{(X_i,Y_i)\}_{i=1}^n$. The loss function in this case is given by 
$$
\ell(v,y):=\rho_\tau(y-v),
$$
where $\rho_\tau(u)=(\tau-\mathbb{I}(u<0))u, u\in\R$ denotes the check function for the quantile $\tau$. Meanwhile, by some calculations, we know the quantile function $m(x)$ minimizes the population risk w.r.t. the above $\ell(v,y)$. Namely, we have
$$
m(x)\in \arg min_{h}{\E(\rho_\tau(Y-h(X)))}.
$$
Therefore, we have reasons to believe that the forest estimator in \eqref{final estimator} works well in this problem.

Let us verify Assumption \ref{assump1}-\ref{assump5}. Firstly, we choose $S=\R$ in Assumption \ref{assump1} and it is easy to check that the univariate function $\ell(\cdot,y)$ is convex for any $y\in S$. Secondly, we fix any  $y\in \S$. Then, the loss function $\ell(v,y)$ is also Lipschitz continuous w.r.t. the first variable $v$   with the Lipschitz constant $M_1(v,y):=\max\{\tau,1-\tau\}, \forall v,y,\in\R$. Thirdly, we choose the envelope function by $M_2(v,y):= \max\{\tau,1-\tau\}\cdot (|v|+|y|)$ in Assumption \ref{assump3}. Fourthly, we always suppose $Y$ is a sub-Gaussian random variable to meet the requirement in Assumption \ref{assump_distribution}. Finally, it remains to find the sufficient condition for Assumption \ref{assump5}.

In fact,   the Knight equality in \cite{knight1998limiting} tells us
         $$
            \rho_\tau(u-v)-\rho_\tau(u)=v(\mathbb{I}(u\le 0)-\tau)+\int_{0}^{v}{(\mathbb{I}(u\le s)-\mathbb{I}(u\le 0))ds},
         $$
    from which we get
    \begin{align*}
      R(h)-R(m) & = \E\left[ \rho_\tau(Y-h(X))-\rho_\tau(Y-m(X))\right] \\
       & = \E\left[ (h(X)-m(X))\mathbb{I}(Y-m(X)\le 0)-\tau\right] \\
       & + \E\left[ \int_{0}^{h(X)-m(X)}{ (\mathbb{I}(Y-m(X)\le s)-\mathbb{I}(Y-m(X)\le 0))ds}\right].
    \end{align*}
    Conditional on $X$, we know that the first part of  above inequality equals to zero by using the definition of $m(x)$. Therefore,
    \begin{align}
       R(h)-R(m) & =\E\left[ \int_{0}^{h(X)-m(X)}{ (\mathbb{I}(Y-m(X)\le s)-\mathbb{I}(Y-m(X)\le 0))ds}\right] \nonumber\\
       & =  \E\left[ \int_{0}^{h(X)-m(X)}{ sgn(s)\P[(Y-m(X))\in (0,s)\cup (s,0))|X]} ds\right] \label{g2fdhskjbfk}.
    \end{align}
   
    To illustrate our basic idea clearly, we simply consider a normal case, where $m_1(X):=\E(Y|X)$ is independent of the residual $\varepsilon=Y-\E(Y|X)\sim N(0,1)$ and $\sup_{x\in [0,1]^d}{|m_1(x)|}<\infty$. The generalization of this sub-Gaussian case can be finished by following the spirit. In this normal case, $m(X)$ is equal to the sum of $m_1(X)$ and the $\tau$-th quantile of $\varepsilon$.  Denote by $q_\tau(\varepsilon)\in \R$ the $\tau$-th quantile of $\varepsilon$. Thus, the conditional distribution of $(Y-m(X))|X$ is the same as the distribution of  $\varepsilon-q_\tau(\varepsilon)$. For any $s_0>0$,
    \begin{equation*}
      \int_{0}^{s_0}\P[(Y-m(X))\in (0,s)\cup (s,0))|X] ds =  \int_{0}^{s_0}\P[\varepsilon\in (q_\tau(\varepsilon),q_\tau(\varepsilon)+s)] ds.
    \end{equation*}
    Since $q_\tau(\varepsilon)$ is a fixed number, we assume $s_0$ is a large number later. By the Lagrange mean value theorem,  the following probability bound holds
    \begin{equation*}
      \frac{1}{\sqrt{2\pi}}\exp{(-(|q_\tau(\varepsilon)|+s_0)^2/2)}\cdot s\le\P[\varepsilon\in (q_\tau(\varepsilon),q_\tau(\varepsilon)+s)] \le  \frac{1}{\sqrt{2\pi}} \cdot s.
    \end{equation*}
    Then, 
   \begin{equation*}
       \frac{1}{\sqrt{2\pi}}\exp{(-(|q_\tau(\varepsilon)|+s_0)^2/2)}\cdot \frac{1}{2}s_0^2\le \int_{0}^{s_0}\P[(Y-m(X))\in (0,s)\cup (s,0))|X] ds \le \frac{1}{\sqrt{2\pi}} \cdot \frac{1}{2}s_0^2.
    \end{equation*}
    With the same argument, we also have
       \begin{equation*}
       \frac{1}{\sqrt{2\pi}}\exp{(-(|q_\tau(\varepsilon)|+|s_0|)^2/2)}\cdot \frac{1}{2}s_0^2\le \int_{0}^{s_0}\P[(Y-m(X))\in (0,s)\cup (s,0))|X] ds \le \frac{1}{\sqrt{2\pi}} \cdot \frac{1}{2}s_0^2
    \end{equation*}
    once $s_0<0$.  Therefore, \eqref{g2fdhskjbfk} implies \begin{equation}\label{djhb}
     R(h)-R(m)\le \frac{1}{\sqrt{2\pi}} \cdot \frac{1}{2} \E(h(X)-m(X))^2
    \end{equation}
    and
    \begin{equation}\label{dashjkd}
        R(h)-R(m)\ge \frac{1}{\sqrt{2\pi}}\exp{(-(|q_\tau(\varepsilon)|+\sup_{x\in [0,1]^d}|h(x)-m(x)|)^2/2)}\cdot \frac{1}{2}\E(h(X)-m(X))^2.
    \end{equation}
    Now, we choose $\beta_n\asymp\sqrt{\ln \ln n}$. The combination of \eqref{djhb} and \eqref{dashjkd} implies that the assumption \ref{assump5} holds with $\kappa=2$ and $c=(\ln n)^{-1}$. Meanwhile, those coefficients in general theoretical results are also of polynomial order of $\ln\ln n$, such as $\sqrt{\E( M_2^2(\beta_n,Y))}\asymp \ln\ln n$. The above setting results that the convergence rate of $\varepsilon(\hat{h}_n)$ is $O_p(n^{-\frac{1}{2}\cdot\frac{p}{p+d}})$ up to a polynomial of $\ln n$.

\subsection{Binary classification}
In previous sections, we give several examples of regression. Now, let us discuss another topic related to classification. In this section, we will show that Mondrian forests can be applied in binary classification as long as the chosen loss function is convex. In detail, we assume $Y\in\{1,-1\}$ takes only two labels and $X\in [0,1]^d$ is the explanation vector. It is well known that the Bayes classifier has the minimal classification error and takes the form:
$$ 
  C^{Bayes}(x)=\mathbb{I}(\eta(x)>0.5)-\mathbb{I}(\eta(x)\le 0.5),
$$  
where $\eta(x):=\P(Y=1|X=x)$. However, such a theoretical optimal classifier is not available due to the unknown $\eta(x)$.  In machine learning, the most commonly used loss function in this problem takes the form
$$
\ell(h(x),y):= \phi(-yh(x)),
$$
where $\phi:\R\to [0,\infty)$ is called a non-negative cost function and $h:\R\to\R$ is the goal function we need to learn from the data. The best $h$ is always chosen as the function that minimizes the empirical risk 
$$
\hat{R}(h):= \frac{1}{n}\sum_{i=1}^{n}{\phi(-Y_ih(X_i))}
$$
over a function class. If this minimizer is denoted by $h^*$,  the best classifier is regarded as
$$
{C^*}(x):= \mathbb{I}(h^*(x)>0)-\mathbb{I}(h^*(x)\le 0), x\in [0,1]^d.
$$
Here are some  examples of the loss function $\ell$.
\begin{enumerate}[(i)]

\item Square cost: $\phi_1(v)=(1+v)^2,v\in\R$ suggested in \cite{li2003loss}.

\item Hinge cost: $\phi_2(v)=\max\{1-v,0\}, v\in\R$ that is used in the support vector machine; see \cite{hearst1998support}.

\item Smoothing hinge cost. A problem with the hinge loss
is that direct optimization is difficult due to the discontinuity
in the derivative at $v = 1$. \cite{rennie2005loss} proposed a
smooth version of the Hinge:
\[
\phi_3(v)=\begin{cases}

0.5-v &v\le 0,\\

(1-v)^2/2 & 0<v\le 1,\\

0 & v\ge 1.
\end{cases}
\]

\item Modified square cost: $\phi_4(v)= \max\{1-v,0\}^2, v\in\R$ used in \cite{zhang2001text}. 

\item Logistic  cost:  $\phi_5(v)= \log_2(1+\exp(v)), v\in\R$ applied in  \cite{friedman2000special}.

\item Exponential cost: $\phi_6(v)= \exp(v), v\in\R$, which is used in the famous Adaboost; see \cite{freund1997decision}.

\end{enumerate}

It is obvious that $Y\in\{1,-1\}$ follows Assumption \ref{assump_distribution}. In order to verify Assumption \ref{assump1}-\ref{assump3}, we need to propose the following three conditions on $\phi$:

\begin{enumerate}
  \item[(a)]  $\phi:\R\to [0,\infty)$ is convex.
  \item[(b)]  $\phi(v)$ is piecewise differentiable and the absolute value of each piece $\phi'(v)$ is upper bounded by a polynomial of $|v|$.
  \item[(c)] $\phi(|v|)\le c_1|v|^\gamma+c_2$ for some $\gamma, c_1,c_2>0$.
\end{enumerate}

These conditions are satisfied by $\phi_1-\phi_5$ obviously. We will discuss the case of $\phi_6$ later due to its dramatic increase in speed. When Condition (a) holds, we have
\begin{align*}
  R(\lambda h_1+(1-\lambda)h_2) & = \E(-Y(\lambda h_1(X)+(1-\lambda)h_2(X)))  \\
   & \le \lambda R(h_1)+(1-\lambda) R(h_2)
\end{align*}
for any two functions $h_1,h_2$ and $\lambda\in (0,1)$. This shows that Assumption \ref{assump1} is true. With a slight abuse of notation, let 
$$
M_1(v,y):=\sup_{v_1\in [-v,v]}{|\phi'(v_1)|}
$$
be the maximal value of piecewise function $|\phi'(v)|$. Then, by Lagrange mean value theorem,
$$
\sup_{y\in\{1,-1\},v_1,v_2\in [-v,v]} |\ell(v_1,y)-\ell(v_2,y)|\le M_1(v,1)|v_1-v_2|.
$$
This shows that Assumption \ref{assump2} holds with the Lipschitz constant $M_1(v,1)$. Finally, Condition (c) ensures that Assumption \ref{assump3} holds with 
$$
M_2(v,y):= c_1|v|^\gamma+c_2.
$$
If we take $\beta_n\asymp\ln n$, all the coefficients in Theorem \ref{Th1}:
\begin{equation}\label{coeef2}
  \max\{\beta_n,\sqrt{\E(M_2^2(\beta_n, Y))}\}, \sup_{y\in [-\ln n,\ln n]}{M_1(C,y)}, \sup_{v\in [-\beta_n,\beta_n]}{|\ell(v,\ln n)|}
\end{equation}
diverges not faster than a polynomial of $\ln n$.  These arguments complete the verification of Assumption \ref{assump1}-\ref{assump3} for cost functions $\phi_1-\phi_5$. For the exponential cost $\phi_6$, some direct calculations imply that the assumption \ref{assump1}-\ref{assump3} also holds, and the coefficients in \eqref{coeef2} are $O(\ln n)$ if we take $\beta_n\asymp \ln\ln n$.

Generally speaking, the minimizer $m$ in \eqref{true} is not unique in this classification case; see Lemma 3 in \cite{lugosi2004bayes}. Therefore, it is meaningless to discuss the statistical consistency of forests as shown in Corollary \ref{Th2} or \ref{Corro_consistency}.  However, we can establish weak consistency for Mondrian forests as follows if the cost function is chosen to be $\phi_5$ or $\phi_6$.

\begin{proposition}\label{pro:cla}
  For cost function $\phi_5$ or $\phi_6$, the minimizer $m$ defined in \eqref{true} always exists. Meanwhile,
  \begin{equation*}
  \lim_{n\to\infty} {\E(R(\hat{h}_n))}= R(m).
  \end{equation*}
\end{proposition}

\subsection{Nonparametric density estimation}
\label{densityestimation}

Assume $X$ is a continuous random vector
defined on $[0,1]^d$ and has a density function $f_0(x), x\in [0,1]^d$. Our interest lies in the estimation of
the unknown function $f_0(x)$ based on an i.i.d. sample of $X$, namely data $\D_n=\{X_i\}_{i=1}^n$. Note that any density estimator  has to satisfy two shape requirements that $f_0$ is non-negative, namely, $f_0(x)\ge 0, x\in [0,1]^d$ and
$\int f_0(x) dx = 1$. These two restrictions can be relaxed by transforming. We have the decomposition
$$
 f_0(x)=\frac{\exp(h_0(x))}{\int \exp(h_0(x)) dx}, x\in [0,1]^d,
$$
where $h_0(x)$ is a real function on $[0,1]^d$. The above relationship helps us focus on estimating $h_0(x)$ only, which will be a statistical learning problem without constraint.   On the other hand, this transformation introduces a model identification problem since $h_0 + c$ and $h_0$ give the same density function, where $c\in\R$. To solve this problem, we impose an additional requirement $\int_{[0,1]^d}{h_0(x)}=0$, which guarantees a one-to-one map between
$f_0$ and $h_0$.

In the case of density estimation, the scaled log-likelihood for any function $h(x)$ based on the sampled data $\D_n$ is
$$
  \hat{R}(h):= \frac{1}{n}\left(-\sum_{i=1}^n{h(X_i)}+ \ln \int_{[0,1]^d}{\exp(h(x))dx}\right)
$$
and its population version is 
$$
 R(h)= -\E(h(X))+\ln \int_{[0,1]^d}{\exp{(h(x))}dx}.
$$

With a slight modification, Mondrian forests can also be applied to this problem. Recall the partition $\{\C_{b,\lambda,j}\}_{j=1}^{K_b(\lambda)}$ of the $b$-th Mondrian with stopping time $\lambda$ satisfies
$$
[0,1]^d=\bigcup_{j=1}^{K_b(\lambda)} \C_{b,\lambda,j}\ \ \text{and}\ \ \C_{b,\lambda,j_1}\cap \C_{b,\lambda,j_2}= \varnothing,\ \forall j_1\neq j_2.
$$
For each cell $\C_{b,\lambda,j}$, a constant $\hat{c}_{b,\lambda,j}\in\R$ is used as the predictor of $h(x)$ in this small region. Thus, the estimator of $\eta_0(x)$ based on a single tree has the form
$$
 \hat{h}_{b,n}^{pre}(x)= \sum_{j=1}^{K_b(\lambda)}{\hat{c}_{b,\lambda,j}\cdot \I(x\in\C_{b,\lambda,j})},
$$
where coefficients are obtained by minimizing the empirical risk function,
\begin{align*}
    (\hat{c}_{b,\lambda,1},\ldots, \hat{c}_{b,\lambda,K_b(\lambda)}):= & \arg min_{\substack{c_{b,\lambda,j}\in [-\beta_n,\beta_n]\\j=1,\ldots,K_b(\lambda)}} \frac{1}{n}\sum_{j=1}^{K_b(\lambda)}\sum_{i=1}^n {-c_{b,\lambda,j}\cdot \mathbb{I}(X_i\in \C_{b,\lambda,j} )}\\
    &+\ln \sum_{j=1}^{K_b(\lambda)}\int_{\C_{b,\lambda,j}}{\exp{(c_{b,\lambda,j})}dx}.
\end{align*}
Since the optimized function above is differentiable with respect to parameters $c_{b,\lambda,j}$s, it is not difficult to show that the corresponding minimum can indeed be achieved. To meet the requirement of our restriction, the estimator based on a single tree is revised by
$$
 \hat{h}_{b,n}(x)=\hat{h}_{b,n}^{pre}(x)- \int_{[0,1]^d}{\hat{h}_{b,n}^{pre}(x) dx}.
$$
Finally, by applying the ensemble technique again, the estimator of $h_0$ based on the Mondrian forest is 
\begin{equation}\label{forest_density_estimator}
    \hat{h}_n(x):= \frac{1}{B}\sum_{b=1}^B \hat{h}_{b,n}(x), x\in [0,1]^d.
\end{equation}

Next, we analyze the theoretical properties of $\hat{h}_n$. The only difference between  \eqref{forest_density_estimator} and previous estimators is that \eqref{forest_density_estimator} is obtained by using an additional penalty, namely
$$
Pen(h):=\ln \int_{[0,1]^d}{\exp{(h(x))}dx},
$$
where $h: [0,1]^d\to\R$.  Our theoretical analysis will be revised as follows. Let the pseudo loss function be $\ell^{pse}(v,y):=-v$ with $ v\ge 0, y\in \R$. It is obvious that Assumption \ref{assump1} holds for $\ell^{pse}(v,y)$.  Assumption \ref{assump2} is satisfied with $M_1(v,y)=1$ and Assumption \ref{assump3} is satisfied with $M_2(v,y)=v$. After choosing $\beta_n\asymp \ln\ln n$ and following similar arguments in Theorem \ref{Th1}, we have
\begin{align}
         \E R(\hat{h}_n)-R(h)&\le c_1\frac{\ln n}{\sqrt{n}}(1+\lambda_n)^d+2d^{\frac{3}{2}p}C\cdot \frac{1}{\lambda_n^p}\nonumber\\
         &+\frac{C}{\sqrt{n}}+\E_{\lambda_n}|Pen(h)-Pen(h_n^*(x))|,\label{dsghbkASHBJ}
\end{align}
where $h\in \mathcal{H}^{p,\beta}([0,1]^d,C)$ ($0<p\le 1$) and $h_n^*(x):= \sum_{j=1}^{K_1(\lambda_n)}\I(x\in \C_{1,\lambda_n,j})h(x_{1,\lambda_n,j})$ ($x_{1,\lambda_n,j}$  is the center of cell $\C_{1,\lambda_n,j}$) and $c_1$ is the coefficient in Theorem \ref{Th1} . Therefore, it remains  to bound $\E_{\lambda_n}|Pen(h)-Pen(h_n^*(x))|$.

\begin{lemma}\label{Density_lemma1}
    For any $h(x)\in \mathcal{H}^{p,\beta}([0,1]^d,C)$ with $0<p\le 1$ and $h_n^*(x)$,
    $$
\E_{\lambda_n}|Pen(h)-Pen(h_n^*(x))|\le \exp(2C)\cdot 2^p d^{\frac{3}{2}p}\cdot \left(\frac{1}{\lambda_n}\right)^p.
    $$
\end{lemma}
Choosing $\lambda_n= n^{\frac{1}{2(p+d)}}$, the combination of \eqref{dsghbkASHBJ} and Lemma \ref{Density_lemma1} implies the regret function bound as follows:
$$
 \E R(\hat{h}_n)-R(h)\le \left(c_1\ln\ln n+3d^{\frac{3}{2}p}C+\exp(2C)\cdot 2^p d^{\frac{3}{2}p}\right)\cdot \left(\frac{1}{n}\right)^{\frac{1}{2}\cdot\frac{p}{p+d}},
$$
where $n$ is sufficiently large. 

Let $\|\cdot\|_\infty$ be the supremum norm of a function. To obtain the consistency rate of our density estimator, we need to change Assumption \ref{assump5} by the fact below.

\begin{lemma}\label{Density_lemma2}
    Suppose the true density $f_0(x)$ is bounded away from zero and infinity, namely $c_0<h_0(x)<c_0^{-1}, \forall x\in [0,1]^d$ and $\beta_n\asymp \ln\ln n$.  For any function $h:[0,1]^d\to \R$ with $\|h\|_\infty\le \beta_n$ and $\int_{[0,1]^d}{h(x)dx}=0$, we have
    $$
     c_0\cdot\frac{1}{\ln n}\cdot \E(h(X)-h_0(X))^2\le R(h)-R(h_0)\le c_0^{-1}\cdot\ln n \cdot \E(h(X)-h_0(X))^2.
    $$
\end{lemma}

Then Lemma \ref{Density_lemma2} immediately implies the following consistency result. 
\begin{proposition}

     Suppose the true density $f_0(x)$ is bounded away from zero and infinity, namely $c_0<h_0(x)<c_0^{-1},\ \forall x\in [0,1]^d$. If  $\lambda_n= n^{\frac{1}{2(p+d)}}$ and the true function  $h_0\in \mathcal{H}^{p,\beta}([0,1]^d,C)$ with $0<p\le 1$ satisfying $\int_{[0,1]^d}{h_0(x)dx}=0$, there is $n_{den}\in\mathbb{Z}^+$ such that for $n>n_{den}$,
    $$
  \E(\hat{h}_n(X)-h_0(X))^2\le c_0^{-1}\cdot\ln n \cdot \left(c_1\ln\ln n+3d^{\frac{3}{2}p}C+\exp(2C)\cdot 2^p d^{\frac{3}{2}p}\right)\cdot \left(\frac{1}{n}\right)^{\frac{1}{2}\cdot\frac{p}{p+d}}.
    $$
    
\end{proposition}

\section{Conclusion}
\label{sec:conc}
In this paper, we proposed a general framework for Mondrian forests, which can be used in many statistical or machine-learning problems. These applications include but are not limited to LSE, generalized regression, density estimation, quantile regression, and binary classification. Meanwhile, we studied the upper bound of its regret/risk function and statistical consistency and showed how to use them in the specific applications listed above. The future work can study the asymptotic distribution of this kind of general Mondrian forests as suggested by \cite{cattaneo2023inference}.

\bibliography{ref}

\newpage

\appendix
\section{Proofs}
\label{app:theorem}








This section contains proofs of theoretical results in the paper. Several useful preliminaries and notations are introduced first. Meanwhile, the constant $c$ in this section is always a positive number and will change from line to line in order to simplify notations.

\setcounter{equation}{0}
\setcounter{theorem}{0}
\renewcommand\theequation{A.\arabic{equation}} 
\renewcommand\thetheorem{A\arabic{theorem}} 
\noindent
\begin{definition}[\cite{blumer1989learnability}]
	Let $\mathcal{F}$ be a Boolean function class in which each $f:\mathcal{Z}\to\{0,1\}$ is binary-valued. The growth function of $\mathcal{F}$ is defined by
	$$
	\Pi_\mathcal{F}(m)=\max_{z_1,\ldots,z_m\in \mathcal{Z}}|\{(f(z_1),\ldots,f(z_m)):f\in\mathcal{F}\}|
	$$
	for each positive integer $m\in\mathbb{Z}_+$.
\end{definition}

\begin{definition}[\cite{gyorfi2002distribution}] \label{coveringnumber}
	Let $z_1,\ldots, z_n \in\mathbb{R}^p$ and $z_1^n=\{z_1,\ldots, z_n\}$. Let $\mathcal{H}$ be a class of functions $h:\mathbb{R}^p\to\mathbb{R}$. An $L_q$ $\varepsilon$-cover of $\mathcal{H}$ on $z_1^n$ is a finite set of functions $h_1,\ldots, h_N: \mathbb{R}^p
	\to \mathbb{R}$ satisfying
	$$
	\min_{1\leq j\leq N}{\left( \frac{1}{n}\sum_{i=1}^n{|h(z_i)-h_j(z_i)|^q}\right)^\frac{1}{q}}<\varepsilon,\ \ \forall h\in \mathcal{H}.
	$$
	Then, the $L_q$ $\varepsilon$-cover number of $\mathcal{H}$ on $z_1^n$, denoted by $\mathcal{N}_q(\varepsilon,\mathcal{H},z_1^n)$, is the minimal size of an $L_q$ $\varepsilon$-cover of $\mathcal{H}$ on $z_1^n$. If there exist no finite $L_q$ $\varepsilon$-cover of $\mathcal{H}$, then the above cover number is defined as $\mathcal{N}_q(\varepsilon,\mathcal{H},z_1^n)=\infty$.
\end{definition}

To bound the generalization error, VC class will be introduced in our proof. To make this supplementary material self-explanatory, we first introduce some basic concepts and facts about the VC dimension; see \cite{shalev2014understanding} for more details.

\begin{definition}[\cite{kosorok2008introduction}]
The subgraph of a real function $f:\mathcal{X}\to\mathbb{R}$ is a subset of $\mathcal{X}\times\mathbb{R}$ defined by
$$
C_f=\{ (x,y)\in\mathcal{X}\times\mathbb{R}:f(x)>y\},
$$
where $\mathcal{X}$ is an abstract set.
\end{definition}

\begin{definition}[\cite{kosorok2008introduction}]
Let $\mathcal{C}$ be a collection of subsets of the set  $\mathcal{X}$ and  $\{x_1,\ldots,x_m\}\subset \mathcal{X}$ be an arbitrary set of $m$ points. Define that $\mathcal{C}$ picks out a certain subset $A$ of $\{x_1,\ldots,x_m\}$ if $A$ can be expressed as $C\cap \{x_1,\ldots,x_m\}$ for some $C\in \mathcal{C}$. The collection $\mathcal{C}$ is said to shatter $\{x_1,\ldots,x_m\}$ if each of $2^m$ subsets can be picked
out.
\end{definition}

\begin{definition}[\cite{kosorok2008introduction}]\label{VC}
The VC dimension of the real function class $\mathcal{F}$, where each $f\in\mathcal{F}$ is defined on $\mathcal{X}$, is the largest integer $VC(\mathcal{C})$ such that a set  of points in $\mathcal{X}\times\mathbb{R}$ with size $VC(\mathcal{C})$ is shattered by $\{\mathcal{C}_f, f\in \mathcal{F}\}$. In this paper, we use $VC(\mathcal{F})$ to denote the VC dimension of $\mathcal{F}$.
\end{definition}

\textit{Proof of Theorem \ref{Th1}.} By Assumption \ref{assump1}, the convexity of risk function implies 
$$
\E(R(\hat{h}_n))\le \frac{1}{B} \sum_{b=1}^B \E(R(\hat{h}_{b,n})).
$$
Therefore, we only need to consider the excess risk of a single tree in the following analysis.

In fact, our proof is based on the following decomposition:
\begin{align}
     \E R(\hat{h}_{1,n})-R(h)= &\E (R(\hat{h}_{1,n})- \hat{R}(\hat{h}_{1,n}))+ \E(\hat{R}(\hat{h}_n)-\hat{R}(h))\nonumber\\
      &+\E (\hat{R}(h)-R(h))\nonumber\\
      &:= I+II+III, \nonumber
\end{align}
where I relates to the variance term of Mondrian tree, and II is the approximation error of Mondrian tree to $h\in \mathcal{H}^{p,\beta}([0,1]^d,C)$ and III measures the error when the empirical loss $\hat{R}(h)$ is used to approximate the theoretical one.

\textit{Analysis of Part I}.
Define two classes first.
$$
\mathcal{T}(t):= \{ \text{A Mondrian tree with $t$ leaves by partitioning $[0,1]^d$}\}
$$
$$
 \mathcal{G}(t):= \Big\{ \sum_{j=1}^t{\I(x\in \C_{j})}\cdot c_j: c_j\in\R,\ \C_j\ 's\ \text{are leaves of a tree in $\mathcal{T}(t)$} \Big\}.
$$
Thus, the truncated function class of $ \mathcal{G}(t)$ is given by
$$
 \mathcal{G}(t, z):=\{\tilde{g}(x)=T_{z}g(x) :  g\in\mathcal{G}(t)\},
$$
where the threshold $z>0$. Then, the part $I$ can be bounded as follows.
\begin{align}
    |I|&\le \E_{\pi_{\lambda_n}}\left( \E_{\D_n}\left| \frac{1}{n}\sum_{i=1}^n{\ell(\hat{h}_{1,n}(X_i),Y_i)- \E(\ell(\hat{h}_{1,n}(X),Y)|\D_n)}\right| \Big|\pi_{\lambda_n}\right)   \nonumber\\
    &\le \E_{\pi_{\lambda_n}}\left( \E_{\D_n}\left| \frac{1}{n}\sum_{i=1}^n{\ell(\hat{h}_{1,n}(X_i),Y_i)- \E(\ell(\hat{h}_{1,n}(X),Y)|\D_n)}\right| \Big|\pi_{\lambda_n}\right)   \nonumber\\
    &\le \E_{\pi_{\lambda_n}}\left( \E_{\D_n}\sup_{g\in \mathcal{G}(K(\lambda_n), \beta_n)}\left| \frac{1}{n}\sum_{i=1}^n{\ell(g(X_i),Y_i)- \E(\ell(g(X),Y))}\right| \Big|\pi_{\lambda_n}\right)\nonumber\\ 
    &= \E_{\pi_{\lambda_n}}\left( \E_{\D_n}\sup_{g\in \mathcal{G}(K(\lambda_n), \beta_n)}\left| \frac{1}{n}\sum_{i=1}^n{\ell(g(X_i),Y_i)- \E(\ell(g(X),Y))}\right| \cap \mathbb{I}(A_n)\Big|\pi_{\lambda_n}\right)\nonumber\\ 
    &+ \E_{\pi_{\lambda_n}}\left( \E_{\D_n}\sup_{g\in \mathcal{G}(K(\lambda_n), \beta_n)}\left| \frac{1}{n}\sum_{i=1}^n{\ell(g(X_i),Y_i)- \E(\ell(g(X),Y))}\right| \cap \mathbb{I}(A_n^c)\Big|\pi_{\lambda_n}\right)\nonumber\\
    &:= I_1+I_{2} \label{asdhbjhbjd},
\end{align}
where $A_n:= \{\max_{1\le i\le n} |Y_i|\le  \ln{n}\}$. Next, we need to find the upper bound of $I_1,I_2$ respectively. 

Let us consider $I_1$ first. Make the decomposition of $I_1$ as below.
\begin{align*}
    I_1&\le \E_{\pi_{\lambda_n}}\left( \E_{\D_n}\sup_{g\in \mathcal{G}(K(\lambda_n), \beta_n)}\left| \frac{1}{n}\sum_{i=1}^n{\ell(g(X_i),T_{\ln{n}}Y_i)- \E(\ell(g(X),Y))}\right| \cap \mathbb{I}(A_n)\Big|\pi_{\lambda_n}\right) \\
    &\le \E_{\pi_{\lambda_n}}\left( \E_{\D_n}\sup_{g\in \mathcal{G}(K(\lambda_n), \beta_n)}\left| \frac{1}{n}\sum_{i=1}^n{\ell(g(X_i),T_{\ln{n}}Y_i)- \E(\ell(g(X),Y))}\right| \Big|\pi_{\lambda_n}\right) \\
    &\le \E_{\pi_{\lambda_n}}\left( \E_{\D_n}\sup_{g\in \mathcal{G}(K(\lambda_n), \beta_n)}\left| \frac{1}{n}\sum_{i=1}^n{\ell(g(X_i),T_{\ln{n}}Y_i)- \E(\ell(g(X),T_{\ln{n}}Y))}\right| \Big|\pi_{\lambda_n}\right) \\
    &+  \E_{\pi_{\lambda_n}}\left( \sup_{g\in \mathcal{G}(K(\lambda_n), \beta_n)}\left| \E(\ell(g(X),T_{\ln{n}}Y))- \E(\ell(g(X),Y))\right| \Big|\pi_{\lambda_n}\right)\\
    &:= I_{1,1}+I_{1,2}.
\end{align*}

The part $I_{1,1}$ can be bounded by considering the covering number of the function class 
$$
\mathcal{L}_n:=\{\ell(g(\cdot),T_{\ln n}(\cdot)): g\in \mathcal{G}(K(\lambda_n), \beta_n)\},$$
where $K(\lambda_n)$ denotes the number of regions in the partition that is constructed by the truncated Mondrian process $\pi_{\lambda_n}$ with stopping time $\lambda_n$. Therefore, $K(\lambda_n)$ is a deterministic number once $\pi_{\lambda_n}$ is given.
For any $\varepsilon>0$, recall the definition of the covering number of $\mathcal{G}( K(\lambda_n))$, namely $\mathcal{N}_1(\varepsilon,\mathcal{G}( K(\lambda_n), \beta_n), z_1^n)$   shown in Definition \ref{coveringnumber}. Now, we suppose 
$$\{\eta_1(x),\eta_2(x),\ldots,\eta_J(x)\}$$
is a $\varepsilon/(M_1(\beta_n,\ln n))$-cover of class $\mathcal{G}(K(\lambda_n), \beta_n)$ in $L^1(z_1^n)$ space, where $L^1(z_1^n):=\{f(x): \|f\|_{z_1^n}:= \frac{1}{n}\sum_{i=1}^n|f(z_i)|<\infty\}$ is equipped with norm $\|\cdot\|_{z_1^n}$ and $J\ge 1$. Without loss of generality, we can further assume $|\eta_j(x)|\le \beta_n$ since $\mathcal{G}(K(\lambda_n), \beta_n)$ is upper bounded by $\ln n$. Otherwise, we consider the truncation of $\eta_j(x)$: $T_{\ln n}{\eta_j(x)}$.  According to Assumption
\ref{assump2}, we know for any $g\in \mathcal{G}(K(\lambda_n), \beta_n)$ and $\eta_j(x)$,
$$
|\ell(g(x),T_{\ln Y}y)-\ell(\eta_j(x),T_{\ln Y}y)|\le M_1(\beta_n,\ln n)|g(x)-\eta_j(x)|, \ \ 
$$
where $ x\in [0,1]^d, y\in\R.$ The  above inequality implies that 
$$
\frac{1}{n}\sum_{i=1}^n |\ell(g(z_i),T_{\ln Y}w_i)-\ell(\eta_j(z_i),T_{\ln Y}w_i)|\le M_1(\beta_n,\ln n)\cdot \frac{1}{n}\sum_{i=1}^n |g(z_i)-\eta_j(z_i)|
$$
for any $z_1^n:=(z_1,z_2,\ldots,z_n)\in \R^d\times\cdots\times\R^d$ and $(w_1,\ldots,w_n)\in\R\times\cdots\times\R$. Therefore, we know $\ell(\eta_1(x),T_{\ln n}y),\ldots,\ell(\eta_J(x),T_{\ln n}y)$ is a $\varepsilon$-cover of class $\mathcal{G}(K(\lambda_n), \beta_n)$ in $L^1(v_1^n)$ space, where $v_1^n:=((z_1^T,w_1)^T,\ldots,(z_n^T,w_n)^T)$. In other words, we have
\begin{equation}\label{bgfcshjdbfcb}
    \mathcal{N}_1(\varepsilon, \mathcal{L}_n, v_1^n)\le \mathcal{N}_1\left(\frac{\varepsilon}{M_1(\beta_n, \ln n)},\mathcal{G}( K(\lambda_n), \beta_n), z_1^n\right).
\end{equation}
Note that $\mathcal{G}( K(\lambda_n), \beta_n)$ is a VC class since we have shown $\mathcal{G}( K(\lambda_n))$ is a VC class in  \eqref{4.7}. Furthermore, we know the function in $\mathcal{G}( K(\lambda_n), \beta_n)$ is upper bounded by $\beta_n$. Therefore, we can bound the RHS of \eqref{bgfcshjdbfcb} by using Theorem 7.12 in \cite{sen2018gentle}
\begin{align}
 &\mathcal{N}_1\left(\frac{\varepsilon}{M_1(\beta_n,\ln n)},\mathcal{G}( K(\lambda_n), \beta_n), z_1^n\right)\nonumber\\
 &\le c\cdot VC(\mathcal{G}( K(\lambda_n), \beta_n)) (4e)^{VC(\mathcal{G}( K(\lambda_n), \beta_n))}\left(\frac{\beta_n}{\varepsilon}\right)^{VC(\mathcal{G}( K(\lambda_n), \beta_n))}, \label{hjw1sdf}
 \end{align}
 where the constant $c>0$ is universal. On the other hand, it is not difficult to show 
 \begin{equation}\label{adgsjashHG}
   VC(\mathcal{G}( K(\lambda_n), \beta_n))\le VC(\mathcal{G}( K(\lambda_n))).
 \end{equation}
  Thus, the combination of \eqref{adgsjashHG}, \eqref{hjw1sdf} and \eqref{bgfcshjdbfcb} implies 
 \begin{equation}\label{789nsvbfc}
       \mathcal{N}_1(\varepsilon, \mathcal{L}_n, v_1^n)\le c\cdot VC(\mathcal{G}( K(\lambda_n))) (4e)^{VC(\mathcal{G}( K(\lambda_n)))}\left(\frac{\beta_n}{\varepsilon}\right)^{VC(\mathcal{G}( K(\lambda_n)))}
 \end{equation}
for each $v_1^n$. Note that the class $\mathcal{L}_n$ has an envelop function $M_2(\beta_n, y)$ satisfying $\nu(n):= \sqrt{\E(M_2^2(\beta_n, Y))}<\infty$ by Assumption \ref{assump3}. Construct a series of independent Rademacher variables $\{b_i\}_{i=1}^n$ sharing with the same distribution $\P(b_i=\pm 1)=0.5, i=1,\ldots,n$. Then, the  symmetrization technique (see Lemma 3.12 in \cite{sen2018gentle}) and the Dudley entropy integral (see (41) in \cite{sen2018gentle}) and \eqref{789nsvbfc} imply
\begin{align}
   I_{1,1} &\le  \E_{\pi_{\lambda_n}}\left( \E_{\D_n}\sup_{g\in \mathcal{G}(K(\lambda_n), \beta_n)}\left| \frac{1}{n}\sum_{i=1}^n{\ell(g(X_i),T_{\ln{n}}Y_i)} b_i\right| \Big|\pi_{\lambda_n}\right)\  \nonumber\\
   & \text{(symmetrization technique)} \nonumber \\
   &\le \E_{\pi_{\lambda_n}}\left( \frac{24}{\sqrt{n}} \int_{0}^{\nu(n)} \sqrt{\ln \mathcal{N}_1(\varepsilon, \mathcal{L}_n, v_1^n)} d\varepsilon \Big|\pi_{\lambda_n}\right)  \ \  \qquad \qquad \nonumber\\
   & \text{(Dudley's entropy integral)} \nonumber\\
   &\le \frac{c}{\sqrt{n}}\cdot \E_{\pi_{\lambda_n}}\left(  \int_{0}^{\nu(n)} \sqrt{\ln (1+c\cdot(\beta_n/\varepsilon)^{2VC(\mathcal{G}( K(\lambda_n) )})}d\varepsilon \Big|\pi_{\lambda_n}\right)    \nonumber\\
   & (\text{E.q. \eqref{789nsvbfc}}) \nonumber\\
   &\le \beta_n \cdot \frac{c}{\sqrt{n}} \cdot \E_{\pi_{\lambda_n}}\left( \int_{0}^{\nu(n)/\beta_n} \sqrt{\ln (1+c(1/\varepsilon)^{2VC( \mathcal{G}(K(\lambda_n))}) } d\varepsilon\Big|\pi_{\lambda_n}\right), \label{dghbsajfhbHJKK} 
\end{align}
where we use the fact that  $\pi_\lambda$ is independent to the data set $\D_n$ and $c>0$ is universal. Without loss of generality, we can assume $\nu(n)/\beta_n<1$; otherwise just set $\beta_n'=\max\{ \nu(n),\beta_n\}$ as the new upper bound of the function class $\mathcal{G}( K(\lambda_n), \beta_n)$. Therefore, \eqref{dghbsajfhbHJKK} also implies
\begin{align}
   I_{1,1}  & \le  \max\{\beta_n,\nu(n)\} \cdot \frac{c}{\sqrt{n}} \cdot \E_{\pi_{\lambda_n}}\left( \int_{0}^{1} \sqrt{\ln (1+c(1/\varepsilon)^{2VC( \mathcal{G}(K(\lambda_n))}) } d\varepsilon\Big|\pi_{\lambda_n}\right)\nonumber\\
  &\le \max\{\beta_n,\nu(n)\} \cdot \frac{c}{\sqrt{n}}\cdot \E_{\pi_{\lambda_n}}\left(\sqrt{\frac{VC(\mathcal{G}(K(\lambda_n))}{n}}\right)\cdot \int_{0}^{1}\sqrt{\ln(1/\varepsilon)}d\varepsilon \nonumber\\
   &\le c\cdot\max\{\beta_n,\nu(n)\}  \cdot \E_{\pi_{\lambda_n}}\left(\sqrt{\frac{VC(\mathcal{G}(K(\lambda_n))}{n}}\right). \label{465sad}
\end{align}

 The next task is to find the VC 
 dimension of class $\mathcal{G}(t)$ for each positive integer $t\in\mathbb{Z}_+$, which is summarized  below.

 \begin{lemma}\label{pro_4}
     For each integer $t\in\mathbb{Z}_+$, $VC(\mathcal{G}(t))\le c(d)\cdot t\ln(t)$.
 \end{lemma}
 \begin{proof}
  Recall two defined classes:
$$
\mathcal{T}(t):= \{ \text{A Mondrian tree with $t$ leaves by partitioning $[0,1]^d$}\}
$$
$$
 \mathcal{G}(t):= \Big\{ \sum_{j=1}^t{\I(x\in \C_{j})}\cdot c_j: c_j\in\R,\ \C_j\ 's\ \text{are leaves of a tree in $\mathcal{T}(t)$} \Big\}.
$$
For any Mondrian tree $g_{t}\in \mathcal{G}(t)$ with the form 
    \begin{equation}\label{ODT tree2}
         g_{t}(x)=\sum_{j=1}^{t}\I(x\in\C_j)\cdot c_j,
    \end{equation}
we also use notation $\{\C_1,\ldots,\C_{t};c_1,\ldots,c_{t}\}$ to denote tree \eqref{ODT tree2} later.

    In order to bound $VC( \mathcal{G}(t))$, we need to consider the related Boolean class:
    	$$
	\mathcal{F}_{t}= \{sgn(f(x,y)): f(x,y)=g(x)-y, g\in \mathcal{G}(t)\},
	$$
	where $sgn(v)=1$ if $v\ge 0$ and $sgn(v)=-1$ otherwise. Recall the VC dimension of $\mathcal{F}_{t}$, denoted by $VC( \mathcal{G}(t))$, is the largest integer $m\in\mathbb{Z}_+$ satisfying $2^m\leq\Pi_{ \mathcal{G}(t)}(m)$; see Definition \ref{VC}. Next we  focus  on bounding $\Pi_{ \mathcal{G}(t)}(m)$ for each positive integer $m\in\mathbb{Z}_+$. 

    Define the partition set generated by $\mathcal{G}(t)$:
    $$
\Pa_{t_n}:=\left\{ \{\C_1,\ldots,\C_{t}\}: \{\C_1,\ldots,\C_{t};c_1,\ldots,c_{t_n}\}\in\mathcal{G}(t)\right\}
    $$
    and the maximal partition number of $m$ points cut by $\Pa_{t_n}$:
    \begin{equation}\label{1}
        K(t):=\max_{x_1,\ldots,x_m\in [0,1]^p}Card\left(\left\{ \bigcup_{\Pa\in\Pa_{t}}\{\{x_1,\ldots,x_m\}\cap A: A\in\Pa\}\right\}\right).
    \end{equation}

    Since each Mondrian tree takes constant value on its leaf $\C_j$ $(j=1,\ldots,t)$, thus there is at most $\ell+1$ ways to pick out any fixed points $\{(x_1,y_1),\ldots, (x_\ell,y_\ell)\}\subseteq[0,1]^p\times \R$ that lie in a same leaf. In other words, 
    \begin{align}
        \max_{\substack{(x_j,y_j)\in[0,1]^p\times\R\\ j=1,\ldots,\ell}}Card\Big(\Big\{(sgn(f(x_1,y_1))&,\ldots,sgn(f(x_\ell,y_\ell)))\nonumber\\
        &: f(x,y)=c-y,x\in [0,1]^p, y\in\R, c\in\R\Big\}\Big)\le \ell +1.\label{2}
    \end{align}

    According to the tree structure   in \eqref{ODT tree2}, the growth function of $\mathcal{G}(t)$ satisfies
    \begin{equation}\label{3}
        \Pi_{\mathcal{F}_{t}}(m)\le K(t)\cdot (m+1)^{t}.
    \end{equation}

    Therefore, the left task is to bound $K(t)$ only. In fact, We can use induction method to prove 
    \begin{equation}\label{4}
        K(t)\le \left[ c(p)(3m)^{p+1}\right]^{t},
    \end{equation}
    where the constant $c(p)$ only depends on $p$.  The arguments are given below.

    When $t=1$, the partition generated by $\mathcal{G}({1})$ is $\{[0,1]^p\}$. Thus, $K(1)=1$ and \eqref{4} is satisfied obviously in this case.

    Suppose \eqref{4} is satisfied with $t-1$. Now we consider the case for $t$. Here, we need to use two facts. First, any $g_{t}$ must be grown from a $g_{t-1}$ by splitting one of its leaves; Second, any $g_{t-1}$ can grow to be another $g_{t}$ by partitioning one of its leaves. In conclusion, 
    \begin{equation}\label{-1}
          \mathcal{G}({t})=\{g_{t}:g_{t} \text{\ is\ obtained by splitting one of the leaves of } g_{t-1}\in \mathcal{G}({t-1})\}.
    \end{equation}
    Suppose $x_1^*,\ldots,x_m^*\in [0,1]^p$ maximizes $K(t_n)$. We divide $\mathcal{G}(t-1)$ into $K(t-1)$ groups:
    $$
\mathcal{G}^j(t-1):=\left\{ g_{t-1}\in  \mathcal{G}(t-1): g_{t-1}s\text{ share with a same partition for } \{x_1^*,\ldots,x_m^*\} \right\},
    $$
    where $j=1,2,\ldots, K(t-1)$. Meanwhile, we write the partition for $\{x_1^*,\ldots,x_m^*\}$ that is cut by $ \mathcal{G}^j(t-1)$ as
    \begin{equation}\label{5}
      \underbrace{  \{x^*_{u_1^j},x^*_{u_2^j},\ldots,x^*_{u^j_{m_{1,j}}}\}} _{m_{1,j}\ \text{times}},\ldots, \underbrace{\{x^*_{u^j_{m_{1,j}+\cdots+m_{t_n-2,j}+1}},\ldots,x^*_{u_m}\}}_{m_{t_n,j}\ \text{times}},
    \end{equation}
    where $(u_1^j,\ldots,u_m^j)$ is a permutation of $(1,\ldots,m)$ and the cardinality of above sets are written as $m_{1,j},m_{2,j},\ldots, m_{t_n,j}$. Note that both $(u_1^j,\ldots,u_m^j)$ and $(m_{1,j},m_{2,j},\ldots,m_{t_n,j})$ are determined once the index $j$ is fixed.

    If we have generated a Mondrian tree with $t$ leaves from its parent in $\mathcal{G}^j(t-1)$, the corresponding partition for $\{x_1^*,\ldots,x_m^*\}$ can be obtained by following two steps below
    \begin{enumerate}
        \item Partition one of sets in \eqref{5} by using a hyperplane $\theta^Tx=s$, where $\theta\in\R^p, s\in\R$.
        \item Keep other $t-2$ sets in \eqref{5} unchanged.
    \end{enumerate}
    Denote by $\tilde{K}_j$ the number of partition for $\{x_1^*,\ldots,x_m^*\}$ after following above process. Then, the Sauer–Shelah lemma (see for example Lemma 6.10 in \cite{shalev2014understanding} ) tells us 
    \begin{equation*}
        \tilde{K}_j\le  \sum_{\ell=1}^{t-1}c(p)\left( \frac{m_\ell e}{p+1}\right)^{p+1}\le \sum_{\ell=1}^{t-1}c(p)\left( 3m_\ell\right)^{p+1}
        \le c(p)\cdot (3m)^{p+1}.
    \end{equation*}
    Since  each $\mathcal{G}^j({t{-1}})$ can at most produce $c(p)(3m)^{p+1}$ new partitions of $\{x_1^*,\ldots,x_m^*\}$ based on \eqref{5}, \eqref{-1} and $\mathcal{G}({t-1})=\cup_{j=1}^{K(t-1)}\mathcal{G}^j({t{-1}})$ imply
    $$
   K(t)\le  K(t-1)\cdot c(p) (3m)^{p+1}.
    $$
    Therefore, \eqref{4} also holds for the case of $K(t)$.

According to \eqref{3} and \eqref{4}, we finally have
$$
  \Pi_{\mathcal{F}_{t}}(m)\le [c(p)\cdot (3m)^{p+1}]^{t}\cdot (m+1)^{t}\le (3c(p)m)^{t(p+2)}.
$$
Solving the inequality
	$$
	2^m\le (3c(p)m)^{t(p+2)}
	$$
	by using the basic inequality $\ln x\leq\gamma\cdot x-\ln \gamma-1 $ with  $ x,\gamma >0$ yields
	\begin{equation}
		VC(\mathcal{G}({t}))\le c(p)\cdot t\ln(t), \forall t\ge 2,
	\end{equation}
	where the constant $c(p)>0$ depends on $p$ only. This completes the proof. 
 \end{proof}

Therefore, we know from Lemma \ref{pro_4} and \eqref{465sad}  that
\begin{equation*}
I_{1,1}\le c\cdot \max\{\beta_n,\nu(n)\}\cdot \E_{\pi_{\lambda_n}}\left(\sqrt{\frac{K(\lambda_n)\ln K(\lambda_n)}{n}}\right).
\end{equation*}
By the basic inequality $\ln x\le x^\beta/\beta,\ \forall x\ge 1, \forall \beta>0$, from above inequality we have
\begin{equation}\label{sjnfdk}
   I_{1,1}\le  c\cdot \frac{\max\{\beta_n,\nu(n)\}}{\sqrt{n}}\cdot \E(K(\lambda_n)).
\end{equation}
Next, from Proposition 2 in \cite{mourtada2020minimax}, we know $\E(K(\lambda_n))=(1+\lambda_n)^d$. Finally, we have the following upper bound for $I_{1,1}$
\begin{equation}\label{dasj}
    I_{1,1}\le c\cdot \frac{\max\{\beta_n,\nu(n)\}}{\sqrt{n}}\cdot (1+\lambda_n)^d.
\end{equation}

Then, we bound  the second part $I_{1,2}$ of $I_1$ by following the arguments below.
\begin{align}
  I_{1,2} &\le  \E_{\pi_{\lambda_n}}\left( \sup_{g\in \mathcal{G}(K(\lambda_n), \beta_n)} \E(|\ell(g(X),T_{\ln{n}}Y)- \ell(g(X),Y)|) \Big|\pi_{\lambda_n}\right) \nonumber\\
  &= \E_{\pi_{\lambda_n}}\left( \sup_{g\in \mathcal{G}(K(\lambda_n), \beta_n)} \E(|\ell(g(X),\ln n)- \ell(g(X),Y)|\cdot \mathbb{I}(\{|Y|>\ln n\})) \Big|\pi_{\lambda_n}\right)\nonumber\\
  &\le \E_{\pi_{\lambda_n}}\left( \left(\sup_{g\in \mathcal{G}(K(\lambda_n), \beta_n)} \E(|\ell(g(X),\ln n)- \ell(g(X),Y)|^2) \right)^{\frac{1}{2}}\Big|\pi_{\lambda_n}\right)\P^{\frac{1}{2}}(|Y|>\ln n)\nonumber\\
  &\le \left( \sup_{x\in [-\beta_n,\beta_n]}{\ell^2 (x,\ln n)} + \E(M^2_2(\beta_n, Y))\right)^{\frac{1}{2}} \cdot c\exp(-\ln n \cdot w(\ln n)/c),
\end{align}
where in the third line we use Cauchy-Schwarz inequality and in last line we use Assumption \ref{assump3} and the tail probability of $Y$ given in Assumption \ref{assump_distribution}.

Finally, we end the \textit{Analysis of Part I} by bounding $I_{1,2}$. In fact, this bound can be processed as follows.
\begin{align}
  I_{1,2} &= \E_{\pi_{\lambda_n}}\left( \E_{\D_n}\sup_{g\in \mathcal{G}(K(\lambda_n), \beta_n)}\left| \frac{1}{n}\sum_{i=1}^n{\ell(g(X_i),Y_i)- \E(\ell(g(X),Y))}\right| \cap \mathbb{I}(A_n^c)\Big|\pi_{\lambda_n}\right)\nonumber \\
  &\le \E_{\pi_{\lambda_n}}\left( \E_{\D_n}\sup_{g\in \mathcal{G}(K(\lambda_n), \beta_n)} \left(  \frac{1}{n}\sum_{i=1}^n{\ell(g(X_i),Y_i)+ \E(\ell(g(X),Y))}\right) \cap \mathbb{I}(A_n^c)\Big|\pi_{\lambda_n}\right) \nonumber\\
  &\le\E_{\pi_{\lambda_n}}\left( \E_{\D_n} \left(  \frac{1}{n}\sum_{i=1}^n{M_2(\beta_n,Y_i)+ \E(M_2(\beta_n,Y))}\right) \cap \mathbb{I}(A_n^c)\Big|\pi_{\lambda_n}\right) \ \ (Assumption\ \ref{assump3})\nonumber\\
  &\le 2\cdot \sqrt{\E(M_2^2(\beta_n, Y))}\cdot \P(A_n^c) \label{gjhedgbj123}.
\end{align}
Thus, we only need to find the upper bound of $ \P(A_n^c)$. By Assumption \ref{assump_distribution}, we know
\begin{align}
   \P(A_n^c)&=1- \P\left(\max_{1\leq i\leq n}|Y_i|\leq \ln n\right)
    =1-\left[\P(|Y_1|\leq \ln n)\right]^n
    \leq 1-(1-c\cdot e^{-c\cdot \ln n\cdot w(\ln n)})^n \nonumber\\
    &\le 1-e^{n\cdot \ln(1-c\cdot e^{-c\cdot \ln n\cdot w(\ln n)})} \leq -n\cdot \ln(1-c\cdot e^{-c\cdot \ln{n}\cdot w(\ln n)})
		\leq  c\cdot n\cdot e^{-c\cdot \ln{n}\cdot w(\ln n)}\nonumber\\
        &\le \frac{c}{n}\label{subgaussguanjian} 
\end{align}
when  $n$ is sufficiently large. Therefore, \eqref{gjhedgbj123} and \eqref{subgaussguanjian} imply 
\begin{equation}\label{I2bound}
  I_2\le c\cdot \sqrt{\E(M_2^2(\beta_n, Y))}\cdot \frac{1}{n}.
\end{equation}

\textit{Analysis of Part II}. Recall 
$$
II:= \E\left( \frac{1}{n}\sum_{i=1}^n \ell(\hat{h}_{1,n}(X_i),Y_i)- \frac{1}{n}\sum_{i=1}^n \ell(h(X_i),Y_i)\right),
$$
which relates to the empirical approximation error of Mondrian forests. First, suppose the first truncated Mondrian process with stopping time $\lambda_n$  is given, denoted by $\pi_{1,\lambda_n}$. Under this restriction, the partition of $[0,1]^d$ is determined and not random, which is denoted by $\{\C_{1,\lambda,j}\}_{j=1}^{K_1(\lambda_n)}$. Let
\begin{align}
    \Delta_n &:= \E_{\D_n}\left( \frac{1}{n}\sum_{i=1}^n \ell(\hat{h}_{1,n}(X_i),Y_i)- \frac{1}{n}\sum_{i=1}^n \ell(h(X_i),Y_i)\right)\nonumber\\
    &=\E_{\D_n} \bigg( \frac{1}{n}\sum_{i=1}^n \ell(\hat{h}_{1,n}(X_i),Y_i)- \frac{1}{n}\sum_{i=1}^n \ell(h_{1,n}^*(X_i),Y_i)\nonumber \\
    &+\frac{1}{n}\sum_{i=1}^n \ell(h_{1,n}^*(X_i),Y_i)-\frac{1}{n}\sum_{i=1}^n \ell(h(X_i),Y_i)\bigg),\nonumber
\end{align}
where $h_n^*(x):= \sum_{j=1}^{K_1(\lambda_n)}\I(x\in \C_{1,\lambda_n,j})h(x_{1,\lambda_n,j})$ and $x_{1,\lambda_n,j}$  denotes the center of cell $\C_{1,\lambda_n,j}$. We  need to highlight  that $\Delta_n$ depends on $\pi_{1,\lambda_n}$. Since $\hat{h}_{1,n}$ is obtained by
$$
 \hat{c}_{b,\lambda,j}=\arg\min_{z\in [-\beta_n,\beta_n]}{ \sum_{i:X_i\in \C_{b,\lambda,j}}}{ \ell(z,Y_i)},
$$ 
we know
\begin{equation}\label{dhbabd}
    \frac{1}{n}\sum_{i=1}^n \ell(\hat{h}_{1,n}(X_i),Y_i)- \frac{1}{n}\sum_{i=1}^n \ell(h_{1,n}^*(X_i),Y_i)\le 0
\end{equation}
once $\beta_n>C$. At this point, we consider two cases about $\Delta_n$:

\noindent{Case I:} $\Delta_n\le 0$. This case is trivial because we already have $\Delta_n\le 0$.

\noindent{Case II:} $\Delta_n> 0$. In this case, \eqref{dhbabd} implies
\begin{align*}
    \Delta_n&\le  \E_{\D_n}\left| \frac{1}{n}\sum_{i=1}^n \ell(h_{1,n}^*(X_i),Y_i)-\frac{1}{n}\sum_{i=1}^n \ell(h(X_i),Y_i) \right|\\
    &\le  \frac{1}{n} \E_{\D_n}\left( \sum_{i=1}^n\left| \ell(h_{1,n}^*(X_i),Y_i)-\ell(h(X_i),Y_i) \right| \right)\\
    &\le \E_{X,Y} \left( |\ell(h_{1,n}^*(X),Y)-\ell(h(X),Y)|\right).
\end{align*}
Let $ D_\lambda(X)$ be the diameter of the cell that $X$ lies in. By Assumption \ref{assump2}, the above inequality further implies %
\begin{align}
     \Delta_n&\le   \E_{X,Y}(M_1(C,Y)|h_{1,n}^*(X)-h(X)|)\nonumber\\
             &\le   \E_{X,Y}(M_1(C,Y)\cdot C\cdot D_{\lambda_n}(X)^\beta)\nonumber\\
             &\le   C\cdot\E_{X,Y}(M_1(C,Y)\cdot D_{\lambda_n}(X)^\beta)\nonumber\\ 
             &\le    C\cdot\E_{X,Y}( M_1(C,Y)\cdot D_{\lambda_n}(X)^\beta \mathbb{I}(|Y|\le \ln n)+M_1(C,Y)\cdot D_{\lambda_n}(X)^\beta\mathbb{I}(|Y|> \ln n))\nonumber\\
             &\le   C\cdot\E_{X,Y}(\sup_{y\in [-\ln n,\ln n]}{M_1(C,y)} \cdot D_{\lambda_n}(X)^\beta +M_1(C,Y)\cdot d^{\frac{\beta}{2}}\cdot \mathbb{I}(|Y|> \ln n)) \nonumber\\
             &\le   C\cdot\left(\sup_{y\in [-\ln n,\ln n]}{M_1(C,y)}  \cdot  \E_X( D_{\lambda_n}(X)^\beta ) + d^{\frac{\beta}{2}}\cdot \sqrt{\E M^2_1(C,Y)}\cdot  \P^{\frac{1}{2}}(|Y|> \ln n)) \right), \label{hbdb}
\end{align}
where the second line holds because $h$ is a $(p,C)$-smooth function and we use $D_\lambda(X)\le \sqrt{d}\ a.s.$ to get the fifth line and Cauchy-Schwarz inequality in  the sixth line .

Therefore, Case I and \eqref{hbdb} in Case II imply that 
\begin{equation}\label{dad}
   \Delta_n\le C\cdot \left(\sup_{y\in [-\ln n,\ln n]}{M_1(C,y)}  \cdot  \E_X( D_{\lambda_n}(X)^\beta)  + d^{\frac{\beta}{2}}\cdot \sqrt{\E M^2_1(C,Y)}\cdot  \P^{\frac{1}{2}}(|Y|> \ln n)) \right)\ a.s..
\end{equation}
Taking expectation on both sides of \eqref{dad} w.r.t. $\lambda_n$ leads that
\begin{align}
    II &\le C\cdot \left(\sup_{y\in [-\ln n,\ln n]}{M_1(C,y)}  \cdot  \E_{X,\lambda_n}( D_{\lambda_n}(X)^\beta)  + d^{\frac{\beta}{2}}\cdot \sqrt{\E M^2_1(C,Y)}\cdot  \P^{\frac{1}{2}}(|Y|> \ln n)) \right)\nonumber\\
       &\le C\cdot \left( \sup_{y\in [-\ln n,\ln n]}{M_1(C,y)}\cdot  \E_X\E_{\lambda_n}( D_{\lambda_n}(x)^\beta|X=x) + d^{\frac{\beta}{2}} \sqrt{\E M^2_1(C,Y)}\cdot  \P^{\frac{1}{2}}(|Y|> \ln n)) \right)\nonumber\\
       &\le C\cdot \Big(\sup_{y\in [-\ln n,\ln n]}{M_1(C,y)}\cdot \E_X\left[ (\E_{\lambda_n}( D_{\lambda_n}(x)|X=x))^\beta\right] \nonumber\\
       &+ d^{\frac{\beta}{2}} \sqrt{\E M^2_1(C,Y)}\cdot  c\exp(-\ln n \cdot w(\ln n)/c) \Big),\label{asdfs}
\end{align}
where $\beta\in (0,1]$ and in the second line we use the fact that the function $v^\beta,v>0$ is concavity. For any fixed $x\in [0,1]^d$, we can bound $\E_{\lambda_n}( D_{\lambda_n}(x)|X=x)$ by using Corollary 1 in \cite{mourtada2020minimax}. In detail, we have
\begin{align}
    \E_{\lambda_n}( D_{\lambda_n}(x)|X=x) &\le \int_0^\infty{d\left(1+\frac{\lambda_n \delta}{\sqrt{d}}\right)\exp\left(-\frac{\lambda_n \delta}{\sqrt{d}}\right)d\delta}\nonumber\\
    &\le 2d^{\frac{3}{2}}\cdot \frac{1}{\lambda_n}.\label{dasdfs}
\end{align}
Thus, the combination of \eqref{asdfs} and \eqref{dasdfs} imply that 
\begin{equation}\label{dsadq12}
    II\le C\cdot \left(2d^{\frac{3}{2}\beta}\sup_{y\in [-\ln n,\ln n]}{M_1(C,y)} \cdot \lambda_n^{-\beta}  + d^{\frac{\beta}{2}}\cdot \sqrt{\E M^2_1(C,Y)}\cdot  \frac{c}{n} \right).
\end{equation}

\textit{Analysis of Part III}. This part can be bounded by  the central limit theorem.  Since the supremum norm of $h$ is bounded by $C$, we know from Assumption \ref{assump3} that 
$$
\ell(h(x),y)\le \sup_{v\in [-C,C]}\ell(v,y)\le M_2(C,y), \ \forall x\in [0,1]^d,y\in\R
$$
with $\E(M^2_2(C,Y))<\infty$. Thus, $M_2(C,y)$ is an envelop function of $\{h\}$. Note that a single function $h$ consists of a Glivenko-Cantelli class and has VC dimension 1. Thus, the application of equation (80) in \cite{sen2018gentle} implies
\begin{equation}\label{edasdfeasd}
    III:=\E (\hat{R}(h)-R(h))\le \frac{c}{\sqrt{n}}\cdot \sqrt{\E(M^2_2(C,Y))}
\end{equation}
for some universal $c>0$.

Finally, the combination of \eqref{dasj},  \eqref{I2bound}, \eqref{dsadq12} and \eqref{edasdfeasd} completes the proof.   \hfill\(\Box\)\\

\textit{Proof of Corollary \ref{Th2}.} This theorem can be obtained directly by using Theorem \ref{Th1} and Assumption \ref{assump5}.  \hfill\(\Box\) \\


\textit{Proof of Corollary \ref{Corro_consistency}.} The proof starts from the observation that our class $\mathcal{H}^{p,\beta}([0,1]^d,C)$ can be used to approximate any general function. Since $m(x)\in\{ f(x):\E |f|^\kappa(X)<\infty\}$, by density argument we know $m(x)$ can be approximated by a sequence of continuous functions in $L^\kappa$ sense. Thus, we just assume $m(x),x\in [0,1]^d$ is continuous. Define the logistic activation $\sigma_{log}(x)=e^x/(1+e^x),x\in\R$. For any $\varepsilon>0$, by Lemma  16.1 in \cite{gyorfi2002distribution} there is $h_\varepsilon(x)=\sum_{j=1}^J a_{\varepsilon,j}\sigma_{log}(\theta_{\varepsilon,j}^\top x+s_{\varepsilon,j}), x\in [0,1]^d$ with $a_{\varepsilon,j},s_{\varepsilon,j}\in\R$ and $\theta_{\varepsilon,j}\in\R^d$ such that
    \begin{equation}\label{dbkjfcjnbNBsVHBJHKMa}
        \E\left| m(X)-h_\varepsilon(X)\right|^\kappa\le \sup_{x\in [0,1]^d}|m(x)-h_\varepsilon(x)|^\kappa\le\frac{\varepsilon}{3}.
    \end{equation}
Since $h_\varepsilon(x)$ is a continuously differentiable, we know $h_\varepsilon(x)\in \mathcal{H}^{p,\beta}([0,1]^d,C(h_\varepsilon))$, where $C(h_\varepsilon)>0$ depends on $h_\varepsilon$ only. Now we fix such $h_\varepsilon(x)$ and make the decomposition as follows
\begin{align}
     \E R(\hat{h}_{1,n})-R(m)= &\E (R(\hat{h}_{1,n})- \hat{R}(\hat{h}_{1,n}))+ \E(\hat{R}(\hat{h}_n)-\hat{R}(m))\nonumber\\
      &+\E (\hat{R}(m)-R(m))\nonumber\\
      &:= I+II+III.\nonumber
\end{align}
Part I and III can be upper bounded by following similar analysis in Theorem \ref{Th1}. Therefore, under assumptions in our theorem, we know both of these two parts converges to zero as $n\to\infty$. Next, we consider Part II. Note that
\begin{align*}
    \E(\hat{R}(\hat{h}_n)-\hat{R}(m))&= \E(\hat{R}(\hat{h}_n)-\hat{R}(h_\varepsilon)) + \E(R(h_\varepsilon)- R(m))\\
    &\le \E(\hat{R}(\hat{h}_n)-\hat{R}(h_\varepsilon)) + \E|h_\varepsilon(X)- m(X)|^\kappa\\
     &\le \E(\hat{R}(\hat{h}_n)-\hat{R}(h_\varepsilon)) + c\cdot \frac{\varepsilon}{3},
\end{align*}
where in the second line we use Assumption \ref{assump5}. Finally, we only need to consider the behavior of term $\E(\hat{R}(\hat{h}_n)-\hat{R}(h_\varepsilon))$ as $n\to\infty$. This can be done by using the analysis of Part II in the proof for Theorem \ref{Th1}.  Taking $C=C(h_\varepsilon)$ in \eqref{dsadq12}, we  have
\begin{align*}
   \E(\hat{R}(\hat{h}_n)-\hat{R}(h_\varepsilon)) \le & C\cdot \Big(2d^{\frac{3}{2}}\sup_{y\in [-\ln n,\ln n]}{M_1(C,y)} \cdot \lambda_n^{-1}  \nonumber\\
   &+ d^{\frac{1}{2}} \sqrt{2\E M^2_1(C,Y)}\cdot c\exp(-\ln n \cdot w(\ln n)/c) \Big),
\end{align*}
which goes to zero as $n$ increases. In conclusion, we have proved that 
$$
\lim_{n\to\infty} \E(\hat{R}(\hat{h}_n)-R(m))=0.
$$
The above inequality and Assumption \ref{assump5} shows that $\hat{h}_n$ is $L^\kappa$ consistent for the general function $m(x),x\in [0,1]^d$.  \hfill\(\Box\) \\

\textit{Proof of Theorem \ref{thm3}.} Based on Assumption \ref{assump1}, we only need to consider the regret function for 
$\hat{h}^*_{1,n}$. For any $\lambda>0$, by the definition of $\hat{h}^*_{1,n}$ we know
\begin{align}
    \E (\hat{R}(\hat{h}_{1,n}^*))&\le \E (\hat{R}(\hat{h}_{1,n}^*)+ \alpha_{n}\cdot \lambda_{n,1}^* )\nonumber\\
    &\le \E (\hat{R}(\hat{h}_{1,n,\lambda})+ \alpha_{n}\cdot \lambda), \label{sfsd123}
\end{align}
where $\hat{h}_{1,n,\lambda}$ is the estimator based on the process $MP_1(\lambda, [0,1]^d)$. 

On the other hand, we have the decomposition below
\begin{align}
     \E R(\hat{h}_{1,n}^*)-R(h)&= \E (R(\hat{h}_{1,n}^*)- \hat{R}(\hat{h}_{1,n}^*))+ \E(\hat{R}(\hat{h}_{1,n}^*)-\hat{R}(h))\nonumber\\
      &+\E (\hat{R}(h)-R(h)):=I+II+III\label{decompo2}.
\end{align}

Firstly, we bound Part $I$. Recall  $A_n:= \{\max_{1\le i\le n} |Y_i|\le  \ln{n}\}$, which is defined below \eqref{asdhbjhbjd}. Make the decomposition of $I$ as follows.
\begin{align}
 I &= \E ( (R(\hat{h}_{1,n}^*)- \hat{R}(\hat{h}_{1,n}^*))\cap \mathbb{I}(A_n))+ \E ( (R(\hat{h}_{1,n}^*)- \hat{R}(\hat{h}_{1,n}^*))\cap \mathbb{I}(A_n^c)) \nonumber\\
   &:= I_{1,1}+I_{1,2} \label{shjgfgbaskjh22}
\end{align}
The key for bounding  $I_{1,1}$ is to find the upper bound of $\lambda_{n,1}^*$. By the definition of $\hat{h}_{1,n}^*$ and Assumption \ref{assump3}, we know if $A_n$ occurs
$$
\alpha_n\cdot \lambda_{n,1}^*\le Pen(0)\le \sup_{y\in [-\ln n,\ln n]}M_2(\beta_n,y).
$$
Therefore, when $A_n$ happens we have 
\begin{equation*}
    \lambda_{n,1}^*\le \frac{ \sup_{y\in [-\ln n,\ln n]}M_2(\beta_n,y)}{\alpha_n}.
\end{equation*}
Following arguments that we used to bound $I_{1,1}$ in the Proof of Theorem \ref{Th1}, we know
\begin{align}
    |I_{1,1}|&\le c\cdot \frac{\max\{\beta_n,\sqrt{\E(M_2^2(\beta_n, Y))}\}}{\sqrt{n}}\cdot (1+\lambda_{n,1}^*)^d\nonumber\\
    &\le c\cdot \frac{\max\{\beta_n,\sqrt{\E(M_2^2(\beta_n, Y))}\}}{\sqrt{n}}\cdot \left(1+\frac{ \sup_{y\in [-\ln n,\ln n]}M_2(\beta_n,y)}{\alpha_n}\right)^d \label{dasj2}
\end{align}
Next, the way for bounding $I_{1,2}$ in \eqref{shjgfgbaskjh22} is similar to that we used to bound $I_{1,2}$ in the proof of Theorem \ref{Th1}. Namely, we have
\begin{equation}\label{fbjshFBjhSBF}
  |I_{1,2}|\le \left( \sup_{x\in [-\beta_n,\beta_n]}{\ell^2 (x,\ln n)} + \E(M^2_2(\beta_n, Y))\right)^{\frac{1}{2}} \cdot c\exp(-\ln n \cdot w(\ln n)/c).
\end{equation}

Secondly, we use  \eqref{sfsd123} to bound Part $II$ in \eqref{decompo2}. By the definition of $\hat{h}_{1,n}^*$, for any $\lambda>0$ we have
$$
  II:= \E(\hat{R}(\hat{h}_{1,n}^*)-\hat{R}(h))\le \E (\hat{R}(\hat{h}_{1,n,\lambda})-\hat{R}(h)+ \alpha_{n}\cdot \lambda).
$$
Similar to the Proof of Theorem \ref{Th1}, the above inequality implies
\begin{equation}\label{dasxc}
     II \le C\cdot \left(2d^{\frac{3}{2}p}\sup_{y\in [-\ln n,\ln n]}{M_1(C,y)} \cdot\lambda_n^{-p}+d^{\frac{p}{2}}\sqrt{\E(M_1^2(C,Y))}\cdot \frac{c}{n} \right)+ \alpha_{n}\cdot \lambda.
\end{equation}
Since \eqref{dasxc} holds for all $\lambda>0$, taking $\lambda= \left(\frac{1}{\alpha_n}\right)^{1/(p+1)}$ inequality \eqref{dasxc} further implies 
\begin{align}\label{sfhjcsvbdjhbcj}
    II&\le (2C\sup_{y\in [-\ln n,\ln n]}{M_1(C,y)}d^{\frac{3}{2}p}+1)\cdot \left(\alpha_n\right)^{\frac{p}{p+1}}+r_n\nonumber\\
    &\le (2C\sup_{y\in [-\ln n,\ln n]}{M_1(C,y)}d^{\frac{3}{2}p}+1)\cdot \left(\alpha_n\right)^{\frac{p}{2}}+r_n,
\end{align}
where $r_n:=C\cdot d^{\frac{1}{2}p}\sqrt{\E(M_1^2(C,Y))}\cdot \frac{c}{n}$ is caused by the probability tail of $Y$.

Thirdly, we consider Part $III$. The argument for this part is same with that used to obtain \eqref{edasdfeasd}. Namely, we have
\begin{equation}\label{edasdfeasd2}
    III:=\E (\hat{R}(h)-R(h))\le \frac{c}{\sqrt{n}}\cdot \sqrt{\E(M^2_2(C,Y))},
\end{equation}
where $c>0$ is universal and does not depend on $C$.

Finally, the combination of \eqref{sfhjcsvbdjhbcj}, \eqref{edasdfeasd2}, \eqref{dasj2} and \eqref{fbjshFBjhSBF} finishes the proof.  \hfill\(\Box\) \\

\textit{Proof of Proposition \ref{secpro:huber}.} For any function $h:[0,1]^d\to [-\beta_n,\beta_n]$, by Assumption \ref{assump_distribution} we know the event 
$$
 F_n:=\bigcap_{i=1}^n\{ Y_i-h(X_i)\in [-C\ln n,C\ln n]\}
$$
happens with probability larger than $1-cn\exp(-C\ln n \cdot w(C\ln n)/c)$, where $C>0$ is a large number and $c>0$. Denote by $\hat{h}_{n,ols}$ the least square forest estimator in Section \ref{subsec:ols}. Now, we make the decomposition below.
$$
\begin{aligned}
  \E (\hat{h}_n(X)-m(X))^2&= \E[ (\hat{h}_n(X)-m(X))^2|F_n]\P(F_n)+\E[(\hat{h}_n(X)-m(X))^2|F_n^c]\P(F_n^c)\\
  \E (\hat{h}_{n,ols}(X)-m(X))^2&= \E[ (\hat{h}_{n,ols}(X)-m(X))^2|F_n]\P(F_n)+\E[(\hat{h}_{n,ols}(X)-m(X))^2|F_n^c]\P(F_n^c)\\
\end{aligned}
$$
When $F_n$ occurs, it can be seen $\hat{h}_{n,ols}=\hat{h}_{n}$ if $\delta_n=C\ln n$ for some large $C>0$. On the other hand, we have the upper bounds of two risk functions:
$$
      \E (\hat{h}_n(X)-m(X))^2\le (2\beta_n^2+2\E m^2(X)),
    R(\hat{h}_{n,ols})\le (2\beta_n^2+2\E m^2(X)).
$$
The combination of above inequalities leads that 
$$
   | \E (\hat{h}_n(X)-m(X))^2-\E (\hat{h}_{n,ols}(X)-m(X))^2|\le (c\ln^2 n+c)\cdot cn\exp(-C\ln n \cdot w(C\ln n)/c).
$$
By Corollary \ref{Corro_consistency}, we already know $\hat{h}_{n,ols}$ is $L^2$ consistent for any general $m(X)$. Therefore, $\hat{h}_n$ is also $L^2$ consistent.  \hfill\(\Box\) \\

\textit{Proof of Proposition \ref{pro:cla}.}
Let us show the existence of $m$ in \eqref{true}. In fact, \cite{lugosi2004bayes} tells us one of the minimizers is
$$
m_*(x):= \inf_{\alpha\in\R}\{\eta(x)\phi(-\alpha)+(1-\eta(x))\phi(\alpha)\}
$$
when  $\phi$ is a differentiable strictly convex, strictly increasing cost function satisfying $\phi(0)=1, \lim_{v\to-\infty}\phi(v)=0$. Let $\varepsilon>0$ be a given small number. Next, we need to show there is $h_*\in\mathcal{H}^{p,\beta}([0,1]^d,C_p)$ such that
\begin{equation}\label{iai}
  R(h_*)-R(m_*)\le \varepsilon.
\end{equation}
when the cost function is $\phi_5$ or $\phi_6$.

First, we consider $\phi_5$. Since $\sup_{v\in\R}|\phi'(v)|\le 1/\ln2$, we have
$$
 R(h)-R(m)\le \frac{1}{\ln 2} \E|h(X)-m_*(X)|.
$$
Note that $\E|m_*(X)|<\infty$. We can always find a infinite differentiable function $h_*\in\mathcal{H}^{p,\beta}([0,1]^d,C_p)$ s.t. $\E|h(X)-m(X)|<\ln2 \varepsilon$. This completes the proof for \eqref{iai}.

Second, we consider $\phi_6$. The argument for $\phi_5$ above does not work due to the dramatic increase of $e^v$. In this case, we need to define an infinite differentiable function
$$
w(x):=e^{\frac{1}{\|x\|_2^2-1}}\I(\|x\|_2<1), x\in\R.
$$
Based on this mollifier, we consider the weighted average function of $m$:
$$
m_\eta(x):= \int_{\R^d} m_*(x-z)\frac{1}{\eta^d}w\left(\frac{z}{\eta} \right)dz, x\in [0,1]^d,
$$
where we define $m(x)=0$ for any $x\notin [0,1]^d$ and $\eta>0$.  We know $m_\eta$ is an infinite differentiable function in $[0,1]^d$ and 
$$
\sup_{x\in [0,1]^d}{|m_\eta(x)|}\in [0,1].
$$
Importantly, some simple analysis implies
\begin{equation}\label{fhuk3jshf}
\lim_{\eta\to 0}{\E|m_\eta(X)-m_*(X)| }=0.
\end{equation}
In fact, we next show one of $m_\eta$ can be defined as $h_*$ satisfying \eqref{iai}. By the dominated convergence theorem, we have 
\begin{align*}
\E(e^{-Ym_*(X)})&=\E\left( \sum_{k=0}^{\infty}\frac{(-Y)^km_*^k(X)}{k!}\right)= \sum_{k=0}^{\infty}\E\left( \frac{(-Y)^km_*^k(X)}{k!}\right)\\
\E(e^{-Ym_\eta(X)})&=\E\left( \sum_{k=0}^{\infty}\frac{(-Y)^km_\eta^k(X)}{k!}\right)= \sum_{k=0}^{\infty}\E\left( \frac{(-Y)^km_\eta^k(X)}{k!}\right)
\end{align*}
Since functions $|m|,|m_\eta|$ are upper bounded by $1$, we can find a $N_{\eta,\varepsilon}\in\mathbb{Z}^+$ such that
\begin{align*}
  \left|R(m_*)-\sum_{k=0}^{N_{\eta,\varepsilon}} \E\left( \frac{(-Y)^km_*^k(X)}{k!}\right)\right| & \le \frac{\varepsilon}{3} \\
  \left|R(m_\eta)-\sum_{k=0}^{N_{\eta,\varepsilon}}\E\left( \frac{(-Y)^km_\eta^k(X)}{k!}\right)\right| & \le \frac{\varepsilon}{3} 
\end{align*}
For any $k\le N_{\eta,\varepsilon}$, we have
\begin{align*}
 |(-Y)^km_*^k(X)-(-Y)^km_\eta^k(X)|&\le |m_*(X)-m_\eta(X)||m_*^{k-1}(X)+m_*^{k-2}(X)m_\eta(X)+\cdots+m_\eta^{k-1}(X) | \\
 &\le k|m_*(X)-m_\eta(X)|.
\end{align*}
From \eqref{fhuk3jshf}, choose $m_{\eta^\varepsilon}$ such that
$$
\E|m_*(X)-m_{\eta^\varepsilon}(X)|\le \left( \sum_{k=1}^{N_{\eta,\varepsilon}} \frac{1}{(k-1)!} \right)^{-1} \frac{\varepsilon}{3}.
$$
Put above inequalities together. Then, we know 
$$
   R(m_{\eta^\varepsilon})-R(m_*)\le \varepsilon.
$$
 Finally,  $m_{\eta^\varepsilon}\in\mathcal{H}^{p,\beta}([0,1]^d,C_p)$ for some $C_p>0$ since it is infinite differentiable. Thus, $m_{\eta^\varepsilon}$ can be $h_*$ defined in \eqref{iai}.

On the other hand, by Remark \ref{remark3} we have
$$
  \varlimsup_{n\to\infty} \E(R(\hat{h}_n))\le R(h_*).
$$
Thus, there exist $N_2\in\mathbb{Z}^+$ such that for any $n\ge N_2$,
$$
 \E(R(\hat{h}_n))\le R(h_*)+\varepsilon.
$$
The proof is completed by combining above inequality and \eqref{iai} together.  \hfill\(\Box\)

\

\textit{Proof of Lemma \ref{Density_lemma1}.} Note that  $Pen(h):=\ln \int_{[0,1]^d}{\exp{(h(x))}dx}$. Define a real function as follows:
$$
 g(\alpha):= Pen((1-\alpha)\cdot h+\alpha\cdot h_n^*),\ 0\le\alpha \le 1.
$$
Thus, we have $g(0)=Pen(h)$ and $g(1)=Pen(h_n^*)$. Later, it will be convenient to use the function $g$ in the analysis of this penalty  function. Since both $h$ and $h_n^*$ are upper bounded, we know $g(\alpha)$ is differentiable and its derivative is
\begin{equation}\label{fhbdjsfbkSJD}
    \frac{d}{d \alpha} g(\alpha)= \frac{\int_{[0,1]^d}{(h_n^*(x)-h(x))\exp{(h(x)+\alpha\cdot(h_n^*(x)-h(x)})dx}}{\int_{[0,1]^d}{\exp{(h(x)+\alpha\cdot(h_n^*(x)-h(x))}dx}}.
\end{equation}
Define a continuous random vector $Z_\alpha$ with the density function
\begin{equation}\label{khfskdhnfkjnsaf}
       f_{Z_\alpha}(x):= \frac{\exp{(h(x)+\alpha\cdot(h_n^*(x)-h(x))}}{\int_{[0,1]^d}{\exp{(h(x)+\alpha\cdot(h_n^*(x)-h(x))}dx}}, \ x\in [0,1]^d.
\end{equation}
From \eqref{fhbdjsfbkSJD} and \eqref{khfskdhnfkjnsaf}, we know
\begin{equation}\label{bojfsoiljfd}
     \frac{d}{d \alpha} g(\alpha)=\E_{Z_\alpha}(h_n^*(Z_\alpha)-h(Z_\alpha)).
\end{equation}
On the other hand, the Lagrange mean theorem implies
\begin{align}
    |Pen(h)-Pen(h_n^*(x))|&=|g(0)-g(1)|\nonumber\\
    &=\left| \frac{d}{d \alpha} g(\alpha)|_{\alpha=\alpha^*} \right|\nonumber\\
    &=\E_{Z_{\alpha^*}}(|h_n^*(Z_{\alpha^*})-h(Z_{\alpha^*})|),\label{HBLJ}
\end{align}
where $\alpha^*\in [0,1]$. Thus, later we only need to consider the last term of \eqref{HBLJ}. Since $f_{Z_{\alpha^*}}(x)\le \exp{(2C)},\ \forall x\in[0,1]^d, \forall \alpha\in [0,1]$, we know from \eqref{HBLJ} that
\begin{equation}
     |Pen(h)-Pen(h_n^*(x))|\le \exp{(2C)}\cdot\E_U(|h_n^*(U)-h(U)|),
\end{equation}
where $U$ follows the uniform distribution in $[0,1]^d$ and is independent with $\pi_\lambda$. By further calculation, we have
\begin{align}
    \E_{\pi_\lambda}|Pen(h)-Pen(h_n^*(x))| &\le \exp (2C)\cdot \E_{\pi_\lambda}\E_{U}(|h_n^*(U)-h(U)|)\nonumber\\
    &\le \exp (2C)\cdot \E_{U} \E_{\pi_\lambda}(C\cdot D_{\lambda_n}(U)^\beta)\nonumber\\
       &\le \exp (2C)\cdot\E_U\left[ (\E_{\lambda_n}( D_{\lambda_n}(u)|U=u))^\beta\right].\nonumber
\end{align}
From \eqref{dasdfs}, we already know $\E_{\lambda_n}( D_{\lambda_n}(u)|U=u)\le 2d^{\frac{3}{2}}\cdot \frac{1}{\lambda_n}$. Thus, above inequality implies
$$
\E_{\pi_\lambda}|Pen(h)-Pen(h_n^*(x))| \le \exp (2C)\cdot 2^\beta d^{\frac{3}{2}\beta}\cdot \left(\frac{1}{\lambda_n}\right)^\beta.
$$
This completes the proof.  \hfill\(\Box\)\\

\textit{Proof of Lemma \ref{Density_lemma2}.} First, we calculate the term  
$\frac{d^2}{d\alpha^2} R(h_0+\alpha g)$, where $\int{g(x)dx}=0$. With some calculation, it is not  difficult to know
\begin{equation}\label{densitysjnfk}
   \frac{d^2}{d\alpha^2} R(h_0+\alpha g) = Var\left( h(X_\alpha)\right),
\end{equation}
where $X_\alpha$ is a continuous random vector in $[0,1]^d$. Furthermore, $X_\alpha$ has the density $f_{X_\alpha}(x)= \exp(g_\alpha(x))/\int {\exp(h_\alpha(x))dx}$ with $g_\alpha(x)=h_0(x)+\alpha g(x)$. Sicne $\|h_0\|_\infty\le c$ and $\|g\|_\infty\le \beta_n$, we can assume without loss generality that $\|g_\alpha\|_\infty\le\beta_n$. This results that 
\begin{equation}\label{js}
   \exp{(-2\beta_n)}\le f_{X_\alpha}(x)\le \exp{(2\beta_n)},\ \forall x\in [0,1]^d.
\end{equation}
Let $U$ follows uniform distribution in $[0,1]^d$.  Equation \eqref{js} implies
\begin{align}
        Var\left( g(X_\alpha)\right)&=\inf_{c>0}\E(g(X_\alpha)-c)^2\nonumber\\
        &\le \exp{(2\beta_n)}\cdot \inf_{c>0}\E(g(U)-c)^2\nonumber\\
        &= \exp{(2\beta_n)}\cdot Var(g(U))=\exp{(2\beta_n)}\cdot \E(g^2(U)).\label{1fsdhbjkf}
\end{align}
With the same argument, we also have
\begin{equation}\label{1fsdhbjkf2}
    Var\left( g(X_\alpha)\right)\ge \exp{(-2\beta_n)}\cdot \E(g^2(U)).
\end{equation}
The combination of \eqref{densitysjnfk}, \eqref{1fsdhbjkf} and \eqref{1fsdhbjkf2} shows that
\begin{equation}\label{VJHVBJHB21}
     c\cdot\exp{(-2\beta_n)}\cdot \E(g^2(X))\le \frac{d^2}{d\alpha^2} R(h_0+\alpha g)\le c^{-1}\cdot\exp{(2\beta_n)}\cdot \E(g^2(X))
\end{equation}
for some universal $c>0$ and any $\alpha\in [0,1]$. Finally, by Taylor expansion, we have
$$
 R(h)=R(h_0)+\frac{d}{d\alpha} R(h_0+\alpha (h-h_0))|_{\alpha=0}+\frac{d^2}{d\alpha^2} R(h_0+\alpha (h-h_0))|_{\alpha=\alpha^*}
$$
for some $\alpha^*\in [0,1]$. Without loss of generality, We  can assume  that $\|h-h_0\|_\infty\le\beta_n$.  Thus, the second derivative $\frac{d^2}{d\alpha^2} R(h_0+\alpha (h-h_0))|_{\alpha=\alpha^*}$ can be bounded by using \eqref{VJHVBJHB21} if we take $g=h-h_0$. Meanwhile, the first derivative satisfies $\frac{d}{d\alpha} R(h_0+\alpha (h-h_0))|_{\alpha=0}=0$  since $h_0$ achieves the minimal value of $R(\cdot)$. Based on above analysis, we have
$$
  c\cdot\frac{1}{\ln n}\cdot \E(h(X)-h_0(X))^2\le R(h)-R(h_0)\le c^{-1}\cdot\ln n \cdot \E(h(X)-h_0(X))^2
$$
for some universal $c>0$. This completes the proof.  \hfill\(\Box\)

\vskip 0.2in

\end{document}